\documentclass[conf]{IEEEtran}
\IEEEoverridecommandlockouts  
\usepackage{lineno}
\usepackage{amsmath}
\usepackage{amscd,amsfonts,amssymb,mathrsfs,pifont,stmaryrd,tipa}
\usepackage{array,cases,dsfont,graphicx,texdraw}
\usepackage{graphics}
\usepackage{epsfig}
\usepackage{cite}
\usepackage{xcolor}
\usepackage{mathbbold}
\usepackage{stfloats}
\usepackage{subfig}
\usepackage{hyperref}
\usepackage{epstopdf}
\usepackage{amsfonts}
\usepackage{enumerate}
\usepackage{amssymb}
\usepackage{amsthm}
\usepackage{graphicx}%

\usepackage{algorithm}
\usepackage{algpseudocode}
\hypersetup{hidelinks}

\makeatletter
\newenvironment{breakablealgorithm}
{
	\begin{center}
		\refstepcounter{algorithm}
		\hrule height.8pt depth0pt \kern2pt
		\renewcommand{\caption}[2][\relax]{
			{\raggedright\textbf{\ALG@name~\thealgorithm} ##2\par}%
			\ifx\relax##1\relax 
			\addcontentsline{loa}{algorithm}{\protect\numberline{\thealgorithm}##2}%
			\else 
			\addcontentsline{loa}{algorithm}{\protect\numberline{\thealgorithm}##1}%
			\fi
			\kern2pt\hrule\kern2pt
		}
	}{
		\kern2pt\hrule\relax
	\end{center}
}
\makeatother

\newtheorem{theorem}{\bf{Theorem}}

\newtheorem{assumption}{{Assumption}}

\newtheorem{definition}{\bf{Definition}}
\newtheorem{example}{\bf{Example}}

\newtheorem{lemma}{\bf{Lemma}}

\newtheorem{remark}{\bf{Remark}}
\theoremstyle{plain}

\allowdisplaybreaks                               

\begin{document}
	
	\title{\huge Distributed Generalized Nash Equilibrium Seeking of  $N$-Coalition Games with Full and Distributive Constraints}
	\author{Chao Sun and Guoqiang Hu\thanks{This study is supported  by A*STAR under its Industry Alignment Fund (LOA Award I1901E0046). C. Sun is with the Continental-NTU Corporate Lab,  Nanyang Technological University, 50 Nanyang Avenue, 639798, Singapore (Email: csun002@e.ntu.edu.sg). G. Hu is with the School of
			Electrical and Electronic Engineering, Nanyang Technological University,
			50 Nanyang Avenue, 639798, Singapore (Email: gqhu@ntu.edu.sg).}}
	\maketitle
	
	\begin{abstract}                          
		In this work, we investigate the distributed generalized Nash equilibrium (GNE) seeking problems for $N$-coalition games with  inequality  constraints. First, we  study the scenario where each agent in a coalition has full information of all the constraints. A finite-time average consensus-based approach is proposed by using full decision information. It is proven that the algorithm converges to a GNE (specifically, a variational equilibrium) of the game. Then, based on the finite-time consensus tracking, we propose a distributed algorithm by using partial decision information only, where only neighboring action information is needed for each agent. Furthermore, we investigate  the scenario where only distributive constraint information is available for the agents. The terminology ``distributive" here refers to the scenario where the constraint information acquisition and processing are conducted across multiple agents in a coalition, and each agent has access to only a subset of constraints, instead of the scenario where each agent has  full information of all the constraints. In this case, the multiplier computation and the constraint information required for each individual are reduced. A full decision information algorithm and a partial decision information algorithm for this scenario are proposed, respectively. Numerical examples are given to verify the proposed algorithms.
	\end{abstract}
	

	\section{Introduction} \label{Sec: Intro}

	The $N$-coalition game, proposed in \cite{ye2017simultaneous}, provides a game framework to analyze the simultaneous cooperation and competition in a networked game. Instead of an individual agent as a player in a general non-cooperative game, the $N$-coalition game generalizes the concept of player to a coalition of networked agents which collaboratively optimize their aggregated payoff function. This game model can be used to model many social, commercial and industrial behaviors. For example, in market competition, multiple small firms combine their own resources and develop the collective ability to compete effectively with large firms. In addition, it can provide a unified formulation for distributed optimization  \cite{gharesifard2013distributed1,9376701} and  Nash equilibrium (NE) seeking of  non-cooperative games \cite{lou2015nash,zhang2019distributed,Romano2020}.
	
	Recently, there has been a growing interest on the $N$-coalition game. The authors in \cite{ye2019unified, pang2020nash,ye2020extremum,pang2020gradient} studied  NE  seeking for unconstrained $N$-coalition games.  In \cite{zeng2019generalized}, the inequality constraint is considered. However, the constraint is only related to one agent in each coalition. To be applied in a  wider range of applications, it is expected to propose a GNE seeking algorithm for more general $N$-coalition games with multiple full constraints that may be related to all players' information.
	
	In addition, all the aforementioned studies use full decision information from the other coalitions instead of only partial (neighboring) decision information. Partial decision information (distributed) algorithms for unconstrained $N$-coalition games were firstly proposed in \cite{ye2017distributed}. However, to the best of our knowledge, the  partial decision information algorithm for constrained $N$-coalition games hasn't been proposed in the literature yet. Note that when the number of agents in each coalition is 1, the distributed GNE seeking problem for constrained $N$-coalition games can be reduced to a distributed GNE seeking problem for generalized noncooperative games (e.g.,  \cite{yi2019operator,  lu2018distributed,franci2020distributed}). Thus, the studied problem is more general.
	
	Motivated by the above considerations, in this work, we study the distributed GNE seeking problem for $N$-coalition games with multiple inequality constraints, which may be related to all agents' information in the game. Two scenarios are considered. In the first scenario, each agent in a coalition has full knowledge of all the constraints. In the second scenario, each agent has access to only a subset of constraints, which is called distributive constraints in this work.  The contributions can be summarized as follows.
	
	1) A finite-time consensus estimation-based approach is proposed for $N$-coalition game with inequality constraints. As a comparison, \cite{ye2017simultaneous,ye2019unified, pang2020nash,ye2020extremum,pang2020gradient} studied NE seeking for unconstrained games while this work considers the  GNE seeking for constrained games, which is more complicated.  Compared with \cite{zeng2019generalized}, the inequality constraints considered in this work are more general in the sense that multiple full constraints that may be related to all players' information are studied. While in \cite{zeng2019generalized}, the inequality constraint is only  related to one agent in each coalition;
	
	2)  Partial decision information algorithms for $N$-coalition games with inequality constraints are proposed, where only neighboring information exchange is required.  Compared with \cite{ye2017distributed}, this work considered constrained games while \cite{ye2017distributed} considered unconstrained games only. In addition, the algorithm given in \cite{zeng2019generalized} uses full decision information from the other coalitions instead of partial decision information. 
	
	3) We consider the scenario where only distributive constraints are available. Each agent only deal with a subset of the constraints, instead of all constraints. This framework may reduce the computation burden of each individual in a coalition,  especially when the gradient computation for   constraints is complicated. To the best of the authors' knowledge, this scenario has not been studied in the literture yet.
	
	The remainder of this work can be summarized as follows: In Section \ref{pre}, notations and preliminary knowledge are given. In Section \ref{pro}, the problem is formulated, and the  $N$-coalition game model is described. In Section  \ref{main3}, the scenario of full constraint information is studied. In Section \ref{2main},  the scenario  of distributive constraint information is studied.  In Section \ref{simu}, numerical examples are provided. Finally, in Section \ref{conclu}, the conclusions of the work are given.
	

	\section{Preliminaries}  \label{pre}
	
	\subsection{Notations\label{Notation}}
	
	Throughout this paper, $\mathbb{R}$, $\mathbb{R}^n$ and   $\mathbb{R}^{n\times n}$ denote the real number set, the $n$-dimensional real vector set, and the $n\times n$ dimensional real matrix set, respectively. $\mathbb{R}_{\geq 0}$ represents the non-negative real number set. $\mathbb{R}_{\geq 0}^n$ represents the $n$-dimensional  non-negative real vector set. $|\cdot|$ is the absolute value. $||\cdot||$ and $||\cdot||_\infty$ represent the Euclidean norm and the infinity norm, respectively. Given $\mathcal{V}=\left\{  1,...,n\right\}$,  $[z_i]_{i\in\mathcal{V}}$ denotes a column vector $[z_1^T,\cdots,z_n^T]^T$. $\text{diag}\{z_{i}, i\in\mathcal{V}\}$ denotes a diagonal matrix with its $i$-th diagonal element being $z_{i}$. Let $e_{ij}\in\mathbb{R}$ with $i,j\in\mathcal{V}$. $[e_{ij}]$ denotes a matrix with its $i$-th row and $j$-th column being $e_{ij}$.   For a  vector $z\in\mathbb{R}^n$, $[z]_i$ is the $i$-th element of $z$. For a set $Q$ and a variable $x$, $\mathcal{P}_{Q}[x]=\text{argmin}_{y\in Q}||x-y||^2$ represents the Euclidean projection of  $x$ onto $Q$. $0$ represents an all-zero vector with an appropriate dimension or the real number zero.  $\mathbf{1}_n$ represents an all-one vector with an appropriate dimension. For a vector $g$, $g\leq 0 $ means that each element of $g$ is less than or equal to zero. $\nabla_x f(z)$ is the gradient of function $f$ with respect to $x$ at point $z$. $\lambda_{min}$ and  $\lambda_{max}$ refer to the minimum eigenvalue and the maximum eigenvalue of a matrix, respectively. $\text{card}\{\cdot\}$ denotes the number of elements in a set.
	
	\subsection{Preliminaries \label{Graph theory}}
	
	Let $\mathcal{G}=\left\{  \mathcal{V},\mathcal{E}\right\}  $ denote an
	undirected graph, where $\mathcal{V=}\left\{  1,...,n\right\}  $ indicates the
	node set and $\mathcal{E\subset V\times V}$ indicates the edge set.
	$\mathcal{V}_{i}=\left\{  j\in\mathcal{V\mid}(j,i)\in\mathcal{E}\right\}  $
	denotes the neighborhood set of node $i$.  Path $\mathcal{P}$ between $v_{0}$ and $v_{k}$\ is the
	sequence $\left\{  v_{0},...,v_{k}\right\}  $ where $(v_{i-1},v_{i}%
	)\in\mathcal{E}$ for $i=1,...,k$ and the nodes are distinct. The number $k$
	is defined as the length of path $\mathcal{P}$. Graph $\mathcal{G}$\ is
	connected if for any two nodes, there is a path in $\mathcal{G}$. The matrix
	$A_d=\left[  a_{ij}\right]  \in%
	\mathbb{R}
	^{n\times n}$ denotes the adjacency matrix of $\mathcal{G}$, where $a_{ij}=1$
	if and only if $(j,i)\in\mathcal{E}$ else $a_{ij}=0$. In this paper, we
	suppose that there is no self loop. The matrix $L\triangleq D-A_d\in%
	\mathbb{R}
	^{n\times n}$ is called the Laplacian matrix of $\mathcal{G}$, where
	$D=\text{diag} { \{d_{i}, i\in\mathcal{V}\}} \in%
	\mathbb{R}
	^{n\times n}$ is a diagonal matrix with $d_{i}=\sum\nolimits_{j=1}^{n}a_{ij}
	$ \cite{RenTAC05}.

	
	A map $F: \mathbb{R}^n\rightarrow \mathbb{R}^n$ is monotone if for all $x,y\in \mathbb{R}^n$, $(F(x)-F(y))^T(x-y)\geq 0$. It is strictly monotone if the inequality is strict when $x\neq y$. $F$ is strongly monotone if there exists a positive constant $m$ such that $(F(x)-F(y))^T(x-y)\geq m||x-y||^2$ for all $x,y\in \mathbb{R}^n$  {\cite{facchinei2007finite}}.
	
	A set $\Omega \subseteq \mathbb{R}^n$ is closed if and only if the limit of every convergent sequence  contained in $\Omega$ is also an element of $\Omega$ {(Theorem 3.2.8 of \cite{stephen})}. Let $x_k\in\mathbb{R}^n$, $k=1,\cdots, \infty$ be a sequence of real vectors. $\bar{x}$ is  an accumulation point {(cluster point)} of sequence $x_k$ if there exists a subsequence $x_{k_s}$ of $x_k$ such that  $\lim_{k_s \to \infty} x_{k_s} = \bar{x}$ {(Page 23 of \cite{topo})}.    For a bounded sequence $x_k\in\mathbb{R}^n$, it contains at least one accumulation point {(Bolzano-Weierstrass
		theorem)}. Furthermore, {if} its accumulation point is unique, {then} the sequence is convergent {to the accumulation point, which can be proven according to the property of limit inferior and limit superior}.
	
	The following lemma will be used in the subsequent convergence analysis.
	
	\begin{lemma} (Projection Operator Properties, Theorem 1.5.5 of \cite{facchinei2007finite}) Let $\Omega$ be a  nonempty, closed and convex subset of $\mathbb{R}^n$. Then, the following conclusions hold: 
		
		\begin{enumerate}
			\item For any $x, y\in\mathbb{R}^n$,
			$(P_\Omega[x]-P_\Omega[y])^T(x-y)\geq||P_\Omega[x]-P_\Omega[y]||^2$ and $||P_\Omega[x]-P_\Omega[y]||\leq ||x-y||$; 
			
			\item $(y-P_\Omega[x])^T(P_\Omega[x]-x)\geq 0$ for any $y\in\Omega$. 
			
		\end{enumerate}\label{projector}
	\end{lemma}

Consensus tracking and average consensus tracking are often applied to estimate or observe unknown signals \cite{chen2017connection}. In this work, the following conclusions on finite-time consensus tracking and finite-time average consensus tracking are used.
	
	\begin{lemma} (Finite-Time Consensus Tracking, Theorem 3.1 of \cite{Cao2012}) Consider a multi-agent system defined on $\mathcal{G}$, which is undirected and connected. Let $x_i \in\mathbb{R}$, $i=1,\cdots, n$ be the state of agent $i$ such that 
		\begin{align}
			\dot{x}_i&=-\alpha\sum_{j\in\mathcal{V}\cup \{0\}}a_{ij}(x_i-x_j) \notag \\
			&-\beta\text{sgn}(\sum_{j\in\mathcal{V}\cup \{0\}} a_{ij} (x_i-x_j)), 
		\end{align}
		where  $\alpha\geq 0$, $\beta>0$, and $x_0$ is a time-varying signal  satisfying $|\dot{x}_0(t)|\leq d$. $a_{i0}=1$ if agent $i$ has access to $x_0$ and $a_{i0}=0$ elsewise. If $\beta> d$, then $x_i(t)=x_0(t)$ for all $t\geq \bar{T}\triangleq \frac{\sqrt{\tilde{x}^T(0)H\tilde{x}(0)}\sqrt{\lambda_{\max}(H)}}{(\beta-d)\lambda_{\min}(H)}$,   where $\tilde{x}$ is a $n$-dimensional vector such that $[\tilde{x}]_i=x_i-x_0$, $i=1,\cdots, n$,  $H=L+B$, $L$ is the Laplacian matrix and $B=\text{diag}\{a_{i0}, i\in\mathcal{V}\}$. Furthermore, for any $i\in\mathcal{V}$,  $|x_i(t)-x_0(t)|\leq \sqrt{\frac{\lambda_{\max}(H)}{\lambda_{\min}(H)}}||\tilde{x}(0)||$ at any $t\geq 0$.  \label{l3}
	\end{lemma}
	
		\begin{lemma} (Finite-Time Average Consensus Tracking, Theorem 1 of \cite{chen2012distributed}) Consider a multi-agent system defined on $\mathcal{G}$, which is undirected and connected. Let $x_i, z_i\in\mathbb{R}$, $i=1,\cdots, n$ be the states of agent $i$ such that 
		\begin{align}
			\dot{z}_i&=\alpha\sum_{j\in\mathcal{V}_i} \text{sgn}(x_j-x_i) \notag \\
			x_i&=z_i+r_i(t),
		\end{align}
		where $z_i(0)=0$, $\alpha>0$ and $r_i(t)$ is a time-varying signal satisfying $|\dot{r}_i(t)|\leq c$ for a certain $c>0$ and all $i=1,\cdots, n$, $t\geq 0$. If $\alpha>c$, then for all $t\geq T\triangleq \frac{1}{2(\alpha-c)}\sum_{i=1}^n\sum_{j\in\mathcal{V}_i}|x_i(0)-x_j(0)|$, $x_i=x_j=\frac{1}{n}\sum_{j=1}^nr_j(t)$. \label{l2}
	\end{lemma}
	
	\section{Problem formulation}  \label{pro}
	
	\subsubsection{Game Description}
	The game considered in this work consists of $N>1$ interacting coalitions that are self-interested to minimize their own objective
	functions. The objective function of coalition $i\in \mathcal{N}\triangleq\{1,\cdots, N \}$ is defined by $f_i(\mathbf{x}_i,\mathbf{x}_{-i})=\sum_{j=1}^{n_i} f_{ij}(\mathbf{x}_i,\mathbf{x}_{-i})$, where $f_{ij}(\mathbf{x}_i,\mathbf{x}_{-i})$ represents the objective function of agent $j\in\{1,\cdots, n_i\}\triangleq\mathcal{N}_i$ in coalition $i$,  $n_i$ is a positive integer representing the number of agents in coalition $i$, $\mathbf{x}_i=[x_{ij}]_{j\in\mathcal{N}_i}\in\Omega_i\subseteq \mathbb{R}^{m_i}$ denotes the action vector of agents in coalition $i$, $x_{ij}\in\Omega_{ij}\subseteq\mathbb{R}^{m_{ij}}$ represents the action of agent $j$ in coalition $i$,  $\mathbf{x}_{-i}\in\mathbb{R}^{\sum_{j\in\mathcal{N},j\neq i}m_j}$ denotes the action vector of agents in all the other coalitions except coalition $i$, $m_i$ is the total dimension of the agents in coalition $i$ and $m_{ij}$ is the dimension of agent $j$ in coalition $i$.  Let $\mathbf{x}=[\mathbf{x}_{1}^T, \cdots, \mathbf{x}_{N}^T]^T  \in\Omega\triangleq \Omega_{1}\times \cdots \times \Omega_{N}$ be the total action vector of all the agents and $M=\sum_{i=1}^Nm_i$ be its dimension.
	
	Suppose that each coalition, as a player, is subject to the following shared constraints:  $g(\mathbf{x}_i,\mathbf{x}_{-i})=[g_1(\mathbf{x}_i,\mathbf{x}_{-i}), \cdots, g_p(\mathbf{x}_i,\mathbf{x}_{-i})]^T\in\mathbb{R}^p\leq 0$, where $p$ represents the total number of the constraints. 
	
	In the following development, for notational convenience, we  replace $f_i(\mathbf{x}_i,\mathbf{x}_{-i})$ and $g(\mathbf{x}_i,\mathbf{x}_{-i})$ with $f_i(\mathbf{x})$ and $g(\mathbf{x})$.
	
	\subsubsection{Graph Topology}
	Denote the communication graph in coalition $i$ as $\mathcal{G}_i$ with the vertice set $\mathcal{N}_i$ and the edge set $\mathcal{E}_i$.  The following assumption is required for the subsequent development.
	
	\begin{assumption}
		$\mathcal{G}_i$ is undirected and connected. \label{innercoal}
	\end{assumption}
	  \begin{remark}
	  	This assumption is commonly used in consensus control \cite{chen2012distributed}, distributed optimization \cite{wang2020distributed} and distributed NE/GNE seeking areas \cite{Romano2020}.  
	  \end{remark}	
	
	\subsubsection{Game Assumptions} The following assumptions on the objective functions and the constraints are needed.

	\begin{assumption}
		\label{fassumption} The objective function $f_{i}(\mathbf{x})$ is convex and twice continuously differentiable in  $\mathbf{x}_i$ for all $\mathbf{x}_{-i}$.
	\end{assumption}
	
	\begin{assumption}
		\label{assumption_convex} 1) Each component of $g(\mathbf{x})$ is convex and twice continuously differentiable in $\mathbf{x}$.
		2) The set $\Delta\triangleq\{\mathbf{x}\in\Omega|g(\mathbf{x})\leq 0\}$
		is nonvoid, convex, and compact. 
		3) The Slater's condition holds for $\Delta$. 
	\end{assumption}
	
	The gradient mapping $F(\mathbf{x})\triangleq[\nabla_{\mathbf{x}_i}  f_i(\mathbf{x})]_{i\in\mathcal{N}}$ is required to satisfy the following assumption.
	
	\begin{assumption}
		\label{mono} $F(\mathbf{x})$ is strictly monotone.
	\end{assumption}
	
	\subsubsection{Objective} The objective of this work is to design an algorithm for each agent in the coalitions such that $\mathbf{x}$ tends to a GNE of the game, defined as follows:
	\begin{definition}
		\textbf{(GNE,  \cite{facchinei2007generalized})}   $\mathbf{x}^*=[\mathbf{x}_{1}^{*T}, \cdots, \mathbf{x}_{N}^{*T}]^T$ is called a GNE of the game if	$f_{i}(\mathbf{x}_i,\mathbf{x}^*_{-i}) \leq f_{i}( \mathbf{x}_i^*,\mathbf{x}^*_{-i}),$ $\forall \mathbf{x}_{i}\in\Omega_i\cap\{{\mathbf{x}_i\in\mathbb{R}^{m_i}|} g(\mathbf{x}_i,\mathbf{x}^*_{-i})\leq 0\}, \forall i \in \mathcal{N}.$ \label{ref}
	\end{definition}
	
	In this work, we are interested to find a variational equilibrium (VE) of the game, which is defined as the solution to VI($F(\mathbf{x}), \Delta$), i.e., the problem of finding $\mathbf{x}^*\in\Delta$ such that 
	\begin{align}
		F(\mathbf{x})^T(\mathbf{x}-\mathbf{x}^*)\geq 0, \qquad  \forall \mathbf{x}\in \Delta.
	\end{align}
	
	Based on Theorem 2.3.3 of \cite{facchinei2007finite}, since $F(\mathbf{x}^*)$ is strictly monotone, VI$(F(\mathbf{x}), \Delta)$ has at most one solution. Since $\Delta$ is bounded, according to  Corollary 2.2.5 of \cite{facchinei2007finite},  VI$(F(\mathbf{x}), \Delta)$ has exactly one solution. 
	
	\subsubsection{Some Lemmas}
	
	The following two lemmas are adaptions of the existing conclusions in the literature  according to Assumptions \ref{fassumption}-\ref{mono}. 
	
	\begin{lemma} \label{lemma2}
		(\textit{KKT Condition for VE}) Under Assumptions \ref{fassumption}-\ref{mono}, ${\mathbf{x}^*}\in \Delta $ is a VE  if and only if there exists non-negative Lagrangian multipliers $\lambda^*_s\in\mathbb{R}, s=1,\cdots, p$ such that 
		\begin{align}
			0\in F(\mathbf{x}^*)+\sum_{s=1}^p\lambda_{s}^*G_s(\mathbf{x}^*)+\mathcal{N}_{cone}(\mathbf{x}^*, \Omega), \notag \\
			\lambda_{s}^*g_s(\mathbf{x}^*)=0, g_s(\mathbf{x}^*)\leq 0, \label{kkt2}
		\end{align}
		where $G_s(\mathbf{x})\triangleq [\nabla_{\mathbf{x}_i}   g_s(\mathbf{x}) ]_{i\in\mathcal{N}}$, and $\mathcal{N}_{cone}(\mathbf{x}^*, \Delta)$ is the normal cone of $\Delta$ at $\mathbf{x}^*$ \footnote{ See Section 2.2 and Section 3.3 of \cite{ruszczynski2011nonlinear} for more details on the definition of cones, normal cones, tangent cones and polar cones.}. According to Lemma 2.38 of \cite{ruszczynski2011nonlinear}, (\ref{kkt2}) can also be written as
		\begin{align}
			\mathbf{x}^*=\mathcal{P}_{\Omega}[{\mathbf{x}^*}-(F(\mathbf{x}^*)+\sum_{s=1}^p\lambda_{s}^*G_s(\mathbf{x}^*))], \notag \\
			\lambda_{s}^*g_s(\mathbf{x}^*)=0, g_s(\mathbf{x}^*)\leq 0. \label{kkt}
		\end{align}
	\end{lemma}
	
	\begin{proof}
		See Appendix \ref{ap1}. 
	\end{proof}
	
	Furthermore, according to a similar argument as in Theorem 3.1 of \cite{facchinei2007generalized}, we can obtain the following conclusion:
	
	\begin{lemma}   \label{lemma3}
		(Relationship between VE and GNE) Under Assumptions 	\ref{fassumption}-\ref{mono}, 
		1) if ${\mathbf{x}^*}$ is a VE and thus for some $\lambda^*_s$ (\ref{kkt}) holds, then  ${\mathbf{x}^*}$ is a GNE; 
		2) if ${\mathbf{x}^*}$ is a GNE at which (\ref{kkt}) holds, then ${\mathbf{x}^*}$ is a VE. 
	\end{lemma}	
	
	\begin{proof}
		See Appendix  \ref{ap2}.
	\end{proof}
	
	%
	%
	
	\section{GNE Seeking of Generalized $N$-coalition Games with Full Constraint Information} \label{main3}

	In the full constraint information scenario, as shown in Fig. \ref{fig:constraintfullinfo}, we suppose that each agent in a coalition has full information of all the constraints.

	\begin{figure} [htbp]
		\centering
		\includegraphics[width=0.9\linewidth]{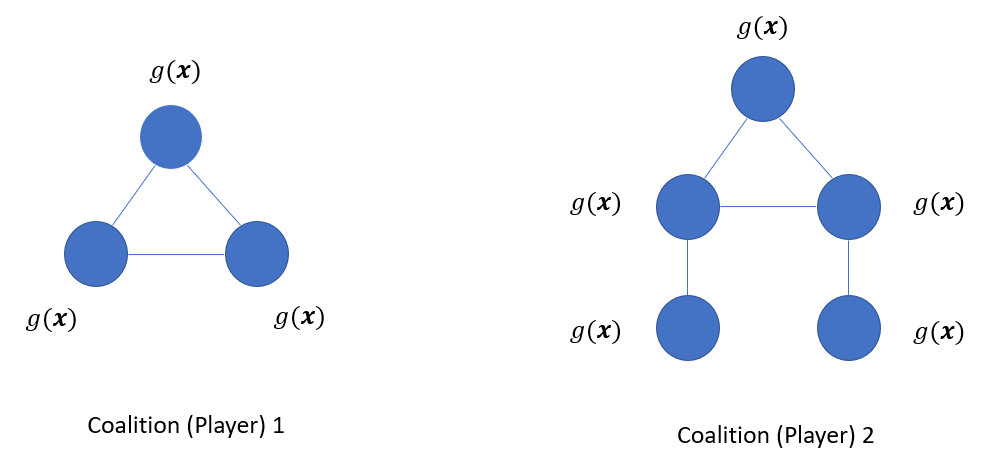}
		\caption{Generalized $N$-coalition games with full constraint information, where $g(\mathbf{x})=[g_1(\mathbf{x}),g_2(\mathbf{x}),g_3(\mathbf{x}),g_4(\mathbf{x}),g_5(\mathbf{x})]^T$. Each agent handles five constraints. }
		\label{fig:constraintfullinfo}
	\end{figure}

	\subsection{Full Decision Information}
	
	In this section, to facilitate a better understanding of the results,  we first design a full decision information algorithm for $N$-coalition games with full constraint information. The analysis and results in this section will be used as a basis for the analysis in the subsequent sections. 
	
	\subsubsection{Algorithm Design}
	
	The  full decision information  algorithm for agent $j$ in coalition $i$ is designed as follows.

	\begin{breakablealgorithm} \label{a2}
		\caption{Full Decision Information A}
		\hspace*{0.02in} \raggedright {\bf Initialization:}  For  $i=1,\cdots, N$, $j,k=1,\cdots, n_i$, $s=1,\cdots, p$, let  $\mathbf{w}_{ij}^k(0)=0$, ${\lambda}_{sij}(0)=1$, and ${\mathbf{x}}_{ij}(0)\in\Omega_{ij}$.
		
		\hspace*{0.02in} \raggedright {\bf Dynamics:} 
		\begin{subequations} \label{al3}
			\begin{align}
				\dot{\mathbf{x}}_{ij}&=\mathcal{P}_{\Omega_{ij}}[{\mathbf{x}}_{ij}-(\mathbf{y}_{ij}^j+\sum_{s=1}^p\lambda_{sij}\nabla_{\mathbf{x}_{ij}} g_s(\mathbf{x}))]-{\mathbf{x}}_{ij}, \label{al3a}\\
				\mathbf{y}_{ij}^k&=\mathbf{w}_{ij}^k+n_i{\nabla_{\mathbf{x}_{ik}} f_{ij}(\mathbf{x})},  \label{al3b}\\
				\dot{\mathbf{w}}_{ij}^k&=-\gamma_i\sum_{r=1}^{n_i} a_{i,jr} \text{sgn}(\mathbf{y}_{ij}^k-\mathbf{y}_{ir}^k), \label{al3c} \\
				\dot{\lambda}_{sij}&=\mathcal{P}_{\mathbb{R}_{\geq 0}}[\lambda_{sij}+g_s(\mathbf{x})]-\lambda_{sij}, \label{al3d}
			\end{align}
		\end{subequations}
	\end{breakablealgorithm}
	where  $a_{i,jr}\in\{0,1\}$ is the element on the $j$-th row and the $r$-th column of the adjacency matrix of $\mathcal{G}_i$, and $\gamma_i$ is a positive gain that is shared in coalition $i$.
	
	To facilitate the subsequent analysis, we define an auxiliary dynamical system:
	\begin{align}
		\dot{\lambda}_{s}&=\mathcal{P}_{\mathbb{R}_{\geq 0}}[\lambda_{s}+g_s(\mathbf{x})]-\lambda_{s}, s=1,\cdots, p, \label{al3r1}
	\end{align}
	with  ${\lambda}_{s}(0)=1\in\mathbb{R}$. 
	
	The initial conditions ${\lambda}_{sij}(0)=1$ and ${\lambda}_{s}(0)=1$ can guarantee that  ${\lambda}_{sij}={\lambda}_{s}$ for all $t\geq 0$.  Thus, the trajectory of $\mathbf{x}_{ij}$ given in (\ref{al3}) is mathematically equivalent to the trajectory of $\mathbf{x}_{ij}$ in the following auxiliary system: 
	\begin{subequations} \label{al4}
		\begin{align}
			\dot{\mathbf{x}}_{ij}&=\mathcal{P}_{\Omega_{ij}}[{\mathbf{x}}_{ij}-(\mathbf{y}_{ij}^j+\sum_{s=1}^p\lambda_{s}\nabla_{\mathbf{x}_{ij}} g_s(\mathbf{x}))]-{\mathbf{x}}_{ij},  \label{al5a} \\
			\mathbf{y}_{ij}^k&=\mathbf{w}_{ij}^k+n_i{\nabla_{\mathbf{x}_{ik}} f_{ij}(\mathbf{x})},   \label{al5b}  \\
			\dot{\mathbf{w}}_{ij}^k&=-\gamma_i\sum_{r=1}^{n_i} a_{i,jr}\text{sgn} (\mathbf{y}_{ij}^k-\mathbf{y}_{ir}^k),  \label{al5c}  \\
			\dot{\lambda}_{s}&=\mathcal{P}_{\mathbb{R}_{\geq 0}}[\lambda_{s}+g_s(\mathbf{x})]-\lambda_{s}. \label{al5e}
		\end{align}
	\end{subequations}

	In  (\ref{al3b}) and (\ref{al3c}), agent $j$ in coalition $i$ updates estimation variables  $\mathbf{y}_{ij}^k$, $k=1,\cdots, n_i$, which use finite-time average consensus \cite{chen2012distributed}\cite{liang2017distributed} to estimate the averaged sum of $n_i\nabla_{\mathbf{x}_{ik}}f_{ij}(\mathbf{x})$ from $j=1$ to $j=n_i$  for $k=1,\cdots, n_i$, which equals to  $\sum_{r=1}^{n_i} \nabla_{\mathbf{x}_{ik}}f_{ir}(\mathbf{x})$. Observing from (\ref{al5a}) and (\ref{al5e}),  (\ref{al3a}) and (\ref{al3d}) are projected primal-dual dynamics using estimation variables at $k=j$ with the estimated value equal to $\sum_{r=1}^{n_i} \nabla_{\mathbf{x}_{ij}}f_{ir}(\mathbf{x})$.
	
	\textbf{Information Exchange of Algorithm  \ref{a2}}: The information exchange of the algorithm can be summarized as follows: 
	
	1) Suppose that $g_s(\mathbf{x}), \nabla_{\mathbf{x}_{ij}} g_s(\mathbf{x}), s=1,\cdots, p$, and ${\nabla_{\mathbf{x}_{ik}} f_{ij}(\mathbf{x})}, k=1,\cdots, n_i$ are directly available for agent $j$ in coalition $i$ at each time instant playing the game.
	
	2) The number of agents in coalition $i$, i.e.,  $n_i$,  and the average consensus gain in coalition $i$, i.e., $\gamma_i$ are known before playing the game for each agent in coalition $i$.
	
	3) Each agent communicates $\mathbf{y}_{ij}^k, k=1,\cdots, n_i$ with its neighbors in $\mathcal{G}_i$.

	\begin{remark}
		Despite the equivalence in mathematics, (\ref{al3}) and (\ref{al4}) are different in practical implementation. The algorithm in (\ref{al4}) requires a centralized coordinator to compute and broadcast the multiplier while  (\ref{al3})  doesn't. In (\ref{al3}), the initial values for multipliers of the same constraint are the same,  which guarantees the equivalence of  (\ref{al3}) and (\ref{al4}), and enables that the multipliers of each shared constraint are identical so that the algorithm can find a VE with analytical convergence.  
	\end{remark}
	
	\subsubsection{Convergence Analysis}
	In the following, we analyze the convergence of the proposed algorithm.
	
	According to \cite{fischer2013lasalle}, since the right-hand side of  (\ref{al3}) is Lebesgue measurable and essentially locally bounded, uniformly in $t$, a Filippov solution to the system exists. In the following sections, we omit the Filippov solution existence statement due to similarity.

	The following lemma guarantees that the action trajectories of the players are bounded, which is important in the convergence analysis.
	
	\begin{lemma}
		Under Assumptions  \ref{innercoal}-\ref{mono} and the initialization condition ${\mathbf{x}(0)}\in\Omega$, the dynamics in (\ref{al5a}) (and (\ref{al3a})) guarantee that ${\mathbf{x}}(t)\in\Omega$ for all $t>0$. \label{lemma1}
	\end{lemma}
	\begin{proof}
		See Appendix  \ref{ap3}. 
	\end{proof}
	
	Similarly, one can obtain the following conclusion which guarantees the non-negativeness of $\lambda_s(t)$.
	
	\begin{lemma}
		Under Assumptions  \ref{innercoal}-\ref{mono} and the initialization condition  ${\lambda}_{s}(0)=0$, the dynamics in (\ref{al5e}) guarantee that $\lambda_s(t)\geq 0$ for all $t\geq 0$. \label{lemma1_1}
	\end{lemma}
	
	According to Lemma \ref{lemma1}, there exist positive constants $ c_{1}$, $c_2$, and $c_3$ such that 
	\begin{align}
		&\Vert \mathbf{x} \Vert\leq c_1, 	\Vert \dot{\mathbf{x}} \Vert\leq c_2, \notag \\
		&\left\Vert
		\frac{d(n_i{\nabla_{\mathbf{x}_{ik}} f_{ij}(\mathbf{x})})}{dt}\right\Vert_\infty\leq c_3,\forall i\in\mathcal{N}, j,k \in\mathcal{N}_i,\label{eqvi}
	\end{align}
	where $c_1$ and $c_2$ are related to an upper bound  of the set $\Omega$, and $c_3$ is related to an upper bound of the set $\Omega$, $n_i$,  and an upper bound  of some Hessian matrix entries within the set $\Omega$.
	
	Then, according to (\ref{eqvi}), the following conclusion  for the dynamics in (\ref{al5b}) and  (\ref{al5c}) holds.
	
	\begin{lemma}
		Under Assumptions  \ref{innercoal}-\ref{mono} and the initialization condition  $\mathbf{w}_{ij}^k(0)=0$, the dynamics in (\ref{al5b}) and  (\ref{al5c})  guarantee that if $\gamma_i>c_3$, then for all $t\geq T_1\triangleq\max_{i\in\mathcal{N}, k\in\mathcal{N}_i}\{ \frac{1}{2(\gamma_i-c_3)}\sum_{j=1}^{n_i}\sum_{r=1}^{n_i}a_{i,jr}\Vert\mathbf{y}_{ij}^k(0)-\mathbf{y}_{ir}^k(0) \Vert_\infty\}$, we have	$\mathbf{y}_{ij}^k(t)=\sum_{r=1}^{n_i}{\nabla_{\mathbf{x}_{ik}} f_{ir}(\mathbf{x})}$. \label{l6}
	\end{lemma}
	\begin{proof}
		It is a direct result of Lemma  \ref{l2}. 
	\end{proof}
	\begin{theorem}
		Under Assumptions  \ref{innercoal}-\ref{mono},  if $\gamma_i>c_3$, Algorithm  \ref{a2} guarantees that each Filippov solution $\mathbf{x}(t)$ converges to the unique VE of the game. \label{theorem1}
	\end{theorem}

	\begin{proof}
		See Appendix  \ref{ap4}.
	\end{proof}

	\begin{remark}
		$\gamma_i$ depends on some global information on the constraint set, and all agents in the same coalition has to agree on its value beforehand. This restriction may be removed by replacing (\ref{al3b}) and (\ref{al3c}) with a fully distributed finite-time average consensus algorithm, while the analysis can remain unchanged due to the finite-time convergence property. 
	\end{remark}

	
	\subsection{Partial Decision Information}  \label{2main}
	
	In this section, motivated by \cite{ye2017distributed}, we design a distributed algorithm such that only partial (neighboring) decision information  is needed. We suppose that  inter-coalition communications are allowed for some agents (see Fig. \ref{fig:constraintdistributed}).
	
	Define the communication graph of all agents in the network as $\mathcal{G}$, where the node set of $\mathcal{G}$ is $\cup_{i\in\mathcal{N}}\mathcal{N}_i$ and the edge set of $\mathcal{G}$ is $\cup_{i\in\mathcal{N}} \mathcal{E}_i\cup \mathcal{E}_{c}$, where $\mathcal{E}_{c}$ is the set of all inter-coalition edges.

	The following assumption is needed for this section.
	
	\begin{assumption}
		\label{coali} There is an inter-coalition path from any coalition to any other coalition.
	\end{assumption}
	
	We can also directly assume that the communication graph $\mathcal{G}$ is connected, which is equivalent to Assumption  \ref{innercoal} plus Assumption \ref{coali}.
	
	\begin{figure} [htbp]
		\centering
		\includegraphics[width=0.7\linewidth]{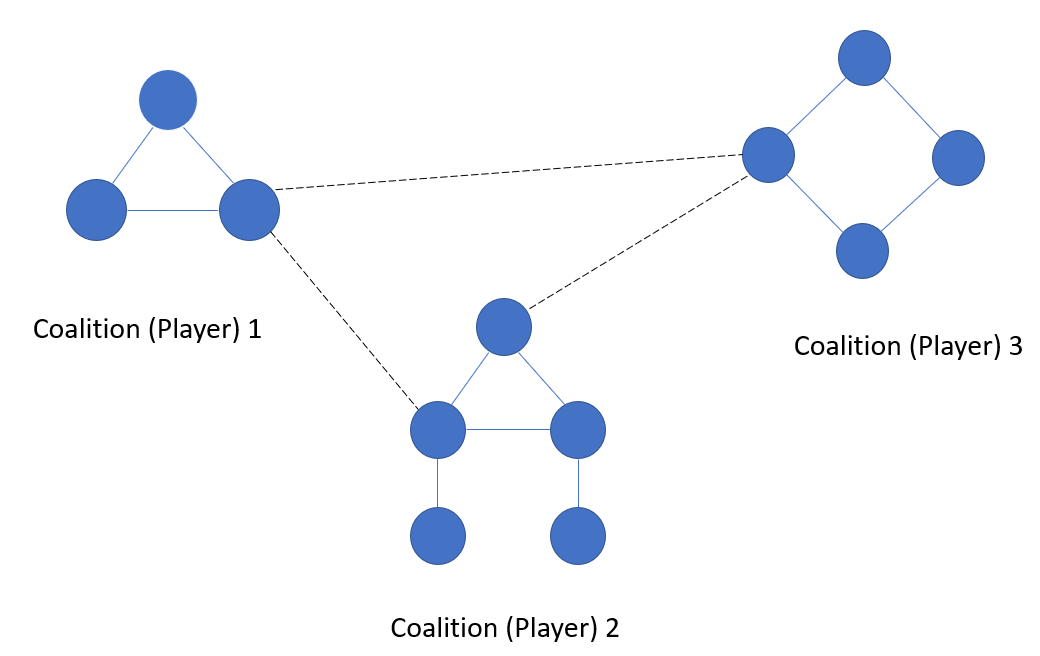}
		\caption{A generalized $N$-coalition game with partial decision information. The dashed line represents the inter-coalition communications.}
		\label{fig:constraintdistributed}
	\end{figure}
	
	For agent $j$ in coalition $i$, we define a subgraph of $\mathcal{G}$  as $\mathcal{G}_{ij}$, where the node set of $\mathcal{G}_{ij}$ is the same with $\mathcal{G}$, and the edge set is defined by replacing the bidirectional edges in $\mathcal{G}$ concerning agent $j$ in coalition $i$ with unidirectional edges starting from  agent $j$ in coalition $i$. The graph $\mathcal{G}_{ij}$ will be used later to estimate the information of agent $j$ in coalition $i$ for all the other agents.
	
	For $b\in\mathcal{N}\neq i$, let $a_{ij,bh}=1$ if there is an edge  in graph $\mathcal{G}$ connecting agent $h\in\mathcal{N}_i$ in coalition $b$ with agent $j$ in coalition $i$, and $a_{ij,bh}=0$ if there is no edge  in graph $\mathcal{G}$ connecting agent $h$ in coalition $b$ with agent $j$ in coalition $i$. If $b=i$, let $a_{ij,bh}=a_{i,jh}$, where $a_{i,jh}$ was defined in the last section. Denote $B_{ij}=\text{diag}\{a_{bh,ij}, {b\in\mathcal{N}, h\in\mathcal{N}_b}\}$. 
	
	Define a subgraph $\bar{\mathcal{G}}_{ij}$ (follower graph) of $\mathcal{G}_{ij}$, where the node set of $\bar{\mathcal{G}}_{ij}$ is $\cup_{i\in\mathcal{N}}\mathcal{N}_i/\{\text{agent $j$ in coalition $i$}\}$ and the edge set of $\bar{\mathcal{G}}_{ij}$ is defined by the edge set of $\mathcal{G}_{ij}$ minus the edge set of edges with agent $j$ in coalition $i$ being a node.
	
	Let $L_{ij}$ be the Laplacian matrix of $\bar{\mathcal{G}}_{ij}$ and define an information exchange matrix $H_{ij}=L_{ij}+B_{ij}$. According to the connectivity of $\mathcal{G}$, agent $j$ in coalition $i$ in $\mathcal{G}_{ij}$ is globally reachable from any  other agent. Then, based on Lemma 4 of \cite{hu2007leader}, $H_{ij}$ is  symmetric and positive definite.
	
	\subsubsection{Algorithm Design}
	
	The following algorithm is designed for the partial decision information scenario.
	
	\begin{breakablealgorithm} \label{a3}
		\caption{Partial Decision Information A}
		\hspace*{0.02in} \raggedright {\bf Initialization:}  For  $i=1,\cdots, N$, $j,k=1,\cdots, n_i$, $s=1,\cdots, p$, let  $\mathbf{w}_{ij}^k(0)=0$, ${\lambda}_{sij}(0)=1$, and ${\mathbf{x}}_{ij}(0)\in\Omega_{ij}$.
		
		\hspace*{0.02in} \raggedright {\bf Dynamics:} 
		\begin{subequations} \label{b1}
			\begin{align}
				\dot{\mathbf{x}}_{ij}&=\mathcal{P}_{\Omega_{ij}}[{\mathbf{x}}_{ij}-(\mathbf{y}_{ij}^j+\sum_{s=1}^p\lambda_{sij}\nabla_{\mathbf{x}_{ij}} g_s(\mathbf{z}_{ij}))]-{\mathbf{x}}_{ij}, \label{bl3a}\\
				\dot{\lambda}_{sij}&=\begin{cases}
					0, & t<T_2,   \\
					\mathcal{P}_{\mathbb{R}_{\geq 0}}[\lambda_{sij}+g_s(\mathbf{z}_{ij})]-\lambda_{sij}, & t\geq T_2, \label{bl3d}
				\end{cases}	 \\
				\mathbf{y}_{ij}^k&=\mathbf{w}_{ij}^k+n_i{\nabla_{\mathbf{x}_{ik}} f_{ij}(\mathbf{z}_{ij})},\label{bl3b}\\
				\dot{\mathbf{w}}_{ij}^k&=-\gamma_i\sum_{r=1}^{n_i} a_{i,jr} \text{sgn}(\mathbf{y}_{ij}^k-\mathbf{y}_{ir}^k), \label{bl3c}  \\
				\dot{\mathbf{z}}_{ij,wq}&=-\alpha\sum_{b\in\mathcal{N},h\in\mathcal{N}_b}a_{ij,bh}(\mathbf{z}_{ij, wq}-\mathbf{z}_{bh,wq}) \notag \\
				-\beta\text{sgn}&(\sum_{b\in\mathcal{N},h\in\mathcal{N}_b}a_{ij,bh}(\mathbf{z}_{ij, wq}-\mathbf{z}_{bh,wq})) \notag \\
				\text { for}& \text{ all }  w\in\mathcal{N}, q\in\mathcal{N}_w \text{ and }  [w,q]\neq [i,j], \label{finite}
			\end{align}
		\end{subequations}
	\end{breakablealgorithm} 
	where $\mathbf{z}_{ij}=[\mathbf{z}_{ij,wq}]_{w\in\mathcal{N}, q\in\mathcal{N}_w}$ with $\mathbf{z}_{ij,ij}=\mathbf{x}_{ij}$, and $T_2$ is a time instant that will be determined later.
	
	Compared with Algorithm \ref{a2}, the idea of Algorithm \ref{a3} is for each agent, we use a finite-time consensus algorithm in (\ref{finite}) to estimate each other agents' action (which is bounded according to the same analysis in Lemma \ref{lemma1}), since only partial decision information is available. Meanwhile, due to the boundedness of the dynamics in  (\ref{finite}),  the dynamics (\ref{bl3b}) and  (\ref{bl3c}) converge to the average of the gradients in a finite time.  After the finite-time convergence of both dynamics,  (\ref{bl3a}) and (\ref{bl3d}) use projected primal-dual dynamics to seek the VE. Note that to guarantee that the multipliers for the same constraint are identical all the time, in (\ref{bl3d}), all the multipliers do not change their values before $T_2$. Here, $T_2$ is the settling time of the finite-time algorithm in (\ref{finite}).
	
	\textbf{Information Exchange of Algorithm  \ref{a3}}: 
	
	1) Suppose that the functions $g_s(\cdot), \nabla_{\mathbf{x}_{ij}} g_s(\cdot), s=1,\cdots, p$ and ${\nabla_{\mathbf{x}_{ik}} f_{ij}(\cdot)}, k=1,\cdots, n_i$ are available for agent $j$ in coalition $i$.
	
	2) The number of agents in coalition $i$, i.e.,  $n_i$,  and the average consensus gain in coalition $i$, i.e., $\gamma_i$ are known for each agent in coalition $i$ before playing the game. The constant $T_2$ and the consensus tracking gains $\alpha$ and $\beta$ are known for all agents in the game. 
	
	3) Each agent communicates $\mathbf{y}_{ij}^k, k=1,\cdots, n_i$ with its neighbors in $\mathcal{G}_i$, and $\mathbf{z}_{ij, wq}, w\in\mathcal{N}, q\in\mathcal{N}_w, [w,q]\neq [i,j]$ with its neighbors in $\mathcal{G}$. 
	
	\subsubsection{Convergence Analysis}
	
	The convergence of the algorithm is illustrated as follows.
	\begin{theorem}
		Suppose that Assumptions  \ref{innercoal}-\ref{coali} hold. Algorithm \ref{a3}  guarantees that each Filippov solution $\mathbf{x}(t)$ converges to the unique VE of the game, provided that $\alpha\geq 0$,  $\beta>c_2$, $T_2\geq \max_{w\in\mathcal{N}, q\in\mathcal{N}_w}\{ {\frac{\lambda_{max}(H_{wq})}{\lambda_{min}(H_{wq})}}\frac{1}{\beta-c_2}\sum_{i=1}^N\sum_{j=1}^{n_i}\Vert\mathbf{z}_{ij, wq}(0)-\mathbf{x}_{wq}(0)\Vert_\infty\}$, and $\gamma_i>c_4$, where $c_2$ was defined in  (\ref{eqvi}), and $c_4$ is a positive constant defined in the proof (in (\ref{d1})). \label{theorem3}
	\end{theorem}
	
	\begin{proof}
		See Appendix \ref{ap6}.
	\end{proof}
	
	\begin{remark}
		The distributed algorithm requires the knowledge of $T_2$, which depends on some global information. In practice, we can select an arbitrary $T_2$ and a sufficiently large $\beta$ such that the condition holds. Fully distributed finite-time consensus algorithms and prescribed-time consensus algorithms may also help to remove the relevant restrictions.
	\end{remark}

	\section{GNE Seeking of Generalized $N$-coalition Games with Distributive Constraint Information}
	
	For the distributive information scenario, as shown in Fig. \ref{fig:constraint}, the information of the $p$ constraint functions are distributively preserved by the agents in a coalition. Compared with the full constraint information model shown in Fig. \ref{fig:constraintfullinfo}, the model in this section is advantageous in the sense that the constraint information acquisition and processing for each individual are reduced.
	
	\begin{figure} [htbp]
		\centering
		\includegraphics[width=0.9\linewidth]{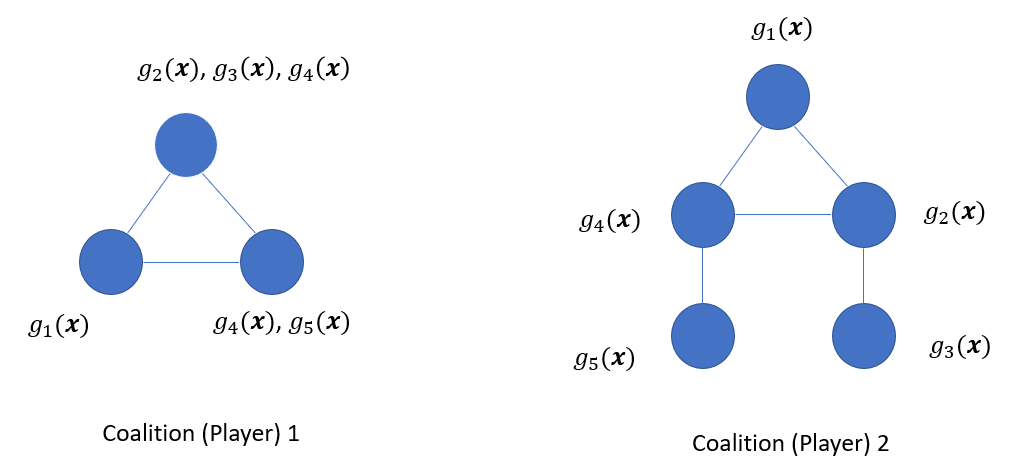}
		\caption{Generalized $N$-coalition games with distributive constraint information. Both players (coalitions) are subject to five constraints. But the constraints are distributively preserved and calculated by the agents in the coalition.}
		\label{fig:constraint}
	\end{figure}

	Denote $p_{ij}'\geq 0$ the number of the constraints owned by agent $j$ in coalition $i$. Let $g_{ij}(\mathbf{x})\in\mathbb{R}^{p_{ij}}$ represent the constraint-related vector of agent $j$ in coalition $i$. If $p_{ij}'=0$, then $p_{ij}=1$ and $g_{ij}(\mathbf{x})\equiv 0\in\mathbb{R}$;  otherwise, $p_{ij}=p_{ij}'$.  
	
	Let $g_{ijl}(\mathbf{x})$ be the $l$-th row ($l=1,\cdots,p_{ij}$) of $g_{ij}(\mathbf{x})$. $g_{ijl}(\mathbf{x})$ ($l=1,\cdots,p_{ij}$) are called  ``constraint-related functions'' to differentiate with the constraints $g_{s}(\mathbf{x})$ ($s=1,\cdots,p$).

	We first propose a full decision information algorithm, followed by a partial decision information algorithm.
	
	\subsection{Full Decision Information}
	
	\subsubsection{Algorithm Design}
	
	Since $g_s(\mathbf{x})$ is unavailable for some agents, Algorithm \ref{a2} cannot be used.
	
	The algorithm can be designed as follows.
	
	\begin{breakablealgorithm} \label{a1}
		\caption{Full Decision Information B}
		\hspace*{0.02in} \raggedright {\bf Initialization:}  For  $i=1,\cdots, N$, $j,k=1,\cdots, n_i$, and $l=1,\cdots, p_{ij}$, let $\mathbf{w}_{ij}^k(0)=0$, ${\lambda}_{ijl}(0)=1$,  and ${\mathbf{x}}_{ij}(0)\in\Omega_{ij}$.
		
		\hspace*{0.02in} \raggedright {\bf Dynamics:} 
		\begin{subequations} \label{al6}
			\begin{align}
				\dot{\mathbf{x}}_{ij}&=\mathcal{P}_{\Omega_{ij}}[{\mathbf{x}}_{ij}-\mathbf{y}_{ij}^j]-{\mathbf{x}}_{ij}, \label{al3a1}\\
				\mathbf{y}_{ij}^k&=\mathbf{w}_{ij}^k+n_i{\nabla_{\mathbf{x}_{ik}} f_{ij}(\mathbf{x})}+\notag \\
				& \sum_{l=1}^{p_{ij}} \frac{n_i}{q_{ijl}} \lambda_{ijl} {\nabla_{\mathbf{x}_{ik}} g_{ijl}(\mathbf{x})},\label{al3b1}\\
				\dot{\mathbf{w}}_{ij}^k&=-\gamma_i\sum_{r=1}^{m_i} a_{i,jr} \text{sgn}(\mathbf{y}_{ij}^k-\mathbf{y}_{ir}^k), \label{al3c1} \\
				\dot{\lambda}_{ijl}&=\mathcal{P}_{\mathbb{R}_{\geq 0}}[\lambda_{ijl}+g_{ijl}(\mathbf{x})]-\lambda_{ijl},  \label{al3d1}
			\end{align}
		\end{subequations}
	\end{breakablealgorithm}
	where $q_{ijl}$ is the number of repeated constraints in coalition $i$, i.e., if  $g_{ijl}(\mathbf{x})=g_{s}(\mathbf{x})$ for some $s\in\{1,\cdots,p\}$, then $q_{ijl}=q_{is}\triangleq\text{card}\{S_{is}\}$ with $S_{is}=\{g_{ijl}(\mathbf{x})| j\in\mathcal{N}_i, l=1,\cdots, p_{ij}, g_{ijl}(\mathbf{x})=g_{s}(\mathbf{x})\}$, and $q_{ijl}=1$ if $g_{ijl}(\mathbf{x})\neq g_{s}(\mathbf{x})$ for all $s\in\{1,\cdots,p\}$.

	Similarly, define an auxiliary system 
	\begin{align}
		\dot{\lambda}_{s}&=\mathcal{P}_{\mathbb{R}_{\geq 0}}[\lambda_{s}+g_s(\mathbf{x})]-\lambda_{s}, s=1,\cdots, p, \label{al3r12}
	\end{align}
	with  ${\lambda}_{s}(0)=1$.
	
	Then, the trajectory of $\mathbf{x}_{ij}$ given in (\ref{al6}) is mathematically equivalent to the trajectory of $\mathbf{x}_{ij}$ in the following auxiliary system: 
	\begin{subequations} \label{al7}
		\begin{align}
			\dot{\mathbf{x}}_{ij}&=\mathcal{P}_{\Omega_{ij}}[{\mathbf{x}}_{ij}-\mathbf{y}_{ij}^j]-{\mathbf{x}}_{ij}, \label{al7a1}\\
			\mathbf{y}_{ij}^k&=\mathbf{w}_{ij}^k+n_i{\nabla_{\mathbf{x}_{ik}} f_{ij}(\mathbf{x})}\notag \\
			&+ \sum_{s=1}^p\frac{n_i}{q_{is}}\delta_{ijs}\lambda_s\nabla_{\mathbf{x}_{ik}} g_{s} (\mathbf{x}),\label{al7b1}\\
			\dot{\mathbf{w}}_{ij}^k&=-\gamma_i\sum_{r=1}^{m_i} a_{i,jr}\text{sgn}(\mathbf{y}_{ij}^k-\mathbf{y}_{ir}^k), \label{al7c1} \\
			\dot{\lambda}_{s}&=\mathcal{P}_{\mathbb{R}_{\geq 0}}[\lambda_{s}+g_s(\mathbf{x})]-\lambda_{s},  \label{al7d1}
		\end{align}
	\end{subequations}
	where $\delta_{ijs}=1$ if $g_s(\mathbf{x})$ is available to agent $j$ in coalition $i$, and $\delta_{ijs}=0$ else-wise.  
	
	The following example shows how to determine the parameter $q_{ijl}$ in Algorithm \ref{a1}.
	
	\begin{example}
		Consider a 2-coalition game where each coalition is subject to two constraints $g_1=A_1\mathbf{x}-b_1$ and $g_2=A_2\mathbf{x}-b_2$. Agent 1 in coalition 1 has no knowledge of any constraint. Thus, $g_{111}\equiv0$. Agent 2 in coalition 1 has access to the constraint $g_1=A_1\mathbf{x}-b_1$ only, which implies that $g_{121}=A_1\mathbf{x}-b_1$. Player 3 in coalition 1 has access to both constraints. Then, $g_{131}=A_1\mathbf{x}-b_1$ and $g_{132}=A_2\mathbf{x}-b_2$. Therefore, $q_{11}=2$ and $q_{12}=1$. Furthermore, $q_{111}=1$,  $q_{121}=2$, $q_{131}=2$ and $q_{132}=1$.  
		
		If the constraints is disordered, for example, $g_{131}=A_2\mathbf{x}-b_2$ and $g_{132}=A_1\mathbf{x}-b_1$, the algorithm is still applicable only if $q_{131} $ and $q_{132} $ are known.   
	\end{example}

	According to  (\ref{al7b1}) and (\ref{al7c1}), we can see that the dynamics  (\ref{al3b1}) and (\ref{al3c1}) use finite-time average consensus \cite{chen2012distributed}\cite{liang2017distributed} to estimate the averaged sum of $n_i\nabla_{\mathbf{x}_{ik}}f_{ij}(\mathbf{x})+\sum_{s=1}^p\frac{n_i}{q_{is}}\delta_{ijs}\lambda_s\nabla_{\mathbf{x}_{ik}} g_{s} (\mathbf{x})$ from $j=1$ to $j=n_i$  for $k=1,\cdots, n_i$, which is equal to $\sum_{r=1}^{n_i} \nabla_{\mathbf{x}_{ik}}f_{ir}(\mathbf{x})+\sum_{s=1}^p\lambda_s\nabla_{\mathbf{x}_{ik}} g_{s} (\mathbf{x})$.   (\ref{al3a1}) and (\ref{al3d1}) are projected primal-dual dynamics using the estimation variables.
	
	\textbf{Information Exchange of Algorithm  \ref{a1}}: 
	
	1) Suppose that $g_{ijl}(\mathbf{x}), \nabla_{\mathbf{x}_{ik}} g_{ijl}(\mathbf{x})$, and ${\nabla_{\mathbf{x}_{ik}} f_{ij}(\mathbf{x})}$, $k=1,\cdots, n_i,l=1,\cdots, p_{ij}$ are directly available for agent $j$ in coalition $i$ at each time instant playing the game.
	
	2) The number of agents in coalition $i$, i.e.,  $n_i$, the average consensus gain in coalition $i$, i.e., $\gamma_i$ are known before playing the game for each agent in coalition $i$. The numbers $q_{ijl}, l=1,\cdots, p_{ij}$ are known for agent $j$ in coalition $i$.
	
	3) Each agent communicates $\mathbf{y}_{ij}^k, k=1,\cdots, n_i$ with its neighbors in $\mathcal{G}_i$.

	\subsubsection{Convergence Analysis}

	A difference between the dynamics in (\ref{al7b1}),  (\ref{al7c1}) and the dynamics in (\ref{al5b}), (\ref{al5c}) is that the average consensus tracked variable in each agent, i.e., $n_i\nabla_{\mathbf{x}_{ik}}f_{ij}(\mathbf{x})+\sum_{s=1}^p\frac{n_i}{q_{is}}\delta_{ijs}\lambda_s\nabla_{\mathbf{x}_{ik}} g_{s} (\mathbf{x})$, does not have a bounded derivative due to the existence of the multipliers. Despite that the multipliers are indeed bounded, we need to prove the boundedness before using it so as to avoid the loop issue. Note that the dynamics in (\ref{al7a1}) and (\ref{al7d1}) are also concerned with  (\ref{al7b1}) and  (\ref{al7c1}). Thus, the boundedness  cannot be directly obtained from (\ref{al7a1}) and (\ref{al7d1}) only.
	
	The sketch of the proof is described as follows: 1)   select an arbitrary time constant $T>0$; 2) calculate the upper bound of the multipliers in $[0, T]$ according to (\ref{al7d1}); 3) select a sufficiently large $\gamma_i$ (concerning $T$) such that  (\ref{al7b1}) and (\ref{al7c1}) are finite-time convergent in $[0, T]$ with a settling time $T_3< T$ based on the fact that the average consensus tracked variables are bounded in $[0, T]$;  4) by contradiction, prove that the average consensus tracked variables are bounded by the bound in $[0, T]$ for all time after $T$; 5) utilizing  Lemma \ref{l2} to prove the theorem.

	The convergence result is shown in the following theorem.
	
	\begin{theorem}
		Suppose that Assumptions  \ref{innercoal}-\ref{mono} hold.  Let $T$ be an arbitrary positive constant. If $\gamma_i>c_5(T)$ and $T_3\triangleq\max_{i\in\mathcal{N}, k\in\mathcal{N}_i}\{ \frac{1}{2(\gamma_i-c_5(T))}\sum_{j=1}^{n_i}\sum_{r=1}^{n_i}a_{i,jr}\Vert\mathbf{y}_{ij}^k(0)-\mathbf{y}_{ir}^k(0)\Vert_\infty\}$ $<T$ where $c_5(T)>0$ is a constant dependent on $T$, then Algorithm  \ref{a1} guarantees that each Filippov solution $\mathbf{x}(t)$ converges to the unique VE of the game.  \label{t2}
	\end{theorem}
	
	\begin{proof}
		See Appendix  \ref{ap5}.
	\end{proof}
	
	\begin{remark}
		It is observed that the parameter $\gamma_i$ satisfying the conditions always exists, since for any fixed $c_5(T)$, we can select $\gamma_i$ to be sufficiently large so that there must exist a solution to the condition. The parameter $c_5(T)$ depends on some global information such as the bound of the whole action space. In practice, one may select a sufficiently large $\gamma_i$ when it is difficult to determine the parameters.
	\end{remark}
	
	\begin{remark}
		A vital difference between Algorithm  \ref{a1} and Algorithm \ref{a2} is the place of the multiplier-related term, which  causes that the selection of parameter $\gamma_i$ in Theorem \ref{t2} is more strict than the selection in Theorem \ref{theorem1}.
	\end{remark}

	\subsection{Partial Decision Information}  
	
	\subsubsection{Algorithm Design}
	
	In this section, we consider the case where distributive constraint information is available. 
	
	The updating law is designed as follows.
	
	\begin{breakablealgorithm} \label{a4}
		\caption{Partial Decision Information B}
		\hspace*{0.02in} \raggedright {\bf Initialization:}  For  $i=1,\cdots, N$, $j,k=1,\cdots, n_i$, and $l=1,\cdots, p_{ij}$, let $\mathbf{w}_{ij}^k(0)=0$, ${\lambda}_{ijl}(0)=1$,  and ${\mathbf{x}}_{ij}(0)\in\Omega_{ij}$.
		
		\hspace*{0.02in} \raggedright {\bf Dynamics:} 
		\begin{subequations} \label{al8}
			\begin{align}
				\dot{\mathbf{x}}_{ij}&=\mathcal{P}_{\Omega_{ij}}[{\mathbf{x}}_{ij}-\mathbf{y}_{ij}^j]-{\mathbf{x}}_{ij}, \label{bl3a1}\\
				\mathbf{y}_{ij}^k&=\mathbf{w}_{ij}^k+n_i{\nabla_{\mathbf{x}_{ik}} f_{ij}(\mathbf{z}_{ij})}+\notag \\
				& \sum_{l=1}^{p_{ij}} \frac{n_i}{q_{ijl}} \lambda_{ijl}  {\nabla_{\mathbf{x}_{ik}} g_{ijl}(\mathbf{z}_{ij})},\label{bl3b1}\\
				\dot{\mathbf{w}}_{ij}^k&=-\gamma_i\sum_{r=1}^{n_i} a_{i,jr} \text{sgn}(\mathbf{y}_{ij}^k-\mathbf{y}_{ir}^k), \label{bl3c1} \\
				\dot{\lambda}_{ijl}& =\begin{cases} 
					0, & t<T_2,  \\
					\mathcal{P}_{\mathbb{R}_{\geq 0}}[\lambda_{ijl}+g_{ijl}(\mathbf{z}_{ij})]-\lambda_{ijl}, & t\geq T_2,  
				\end{cases} \label {lambdaa} \\ 
				\dot{\mathbf{z}}_{ij,wq}&=-\alpha\sum_{b\in\mathcal{N},h\in\mathcal{N}_b}a_{ij,bh}(\mathbf{z}_{ij, wq}-\mathbf{z}_{bh,wq}) \notag \\
				-\beta\text{sgn}&(\sum_{b\in\mathcal{N},h\in\mathcal{N}_b}a_{ij,bh}(\mathbf{z}_{ij, wq}-\mathbf{z}_{bh,wq})) \notag \\
				\text { for} & \text{ all }  w\in\mathcal{N}, q\in\mathcal{N}_w \text{ and }  [w,q]\neq [i,j], \label{finite3}
			\end{align}
		\end{subequations}
		
	\end{breakablealgorithm}
	where $\mathbf{z}_{ij}=[\mathbf{z}_{ij,wq}]_{w\in\mathcal{N}, q\in\mathcal{N}_w}$ with $\mathbf{z}_{ij,ij}=\mathbf{x}_{ij}$, and $T_2$ is a positive constant.

	\textbf{Information Exchange of Algorithm  \ref{a4}}: 
	
	1) Suppose that the functions $g_{ijl}(\cdot)$, $\nabla_{\mathbf{x}_{ik}} g_{ijl}(\cdot)$, and ${\nabla_{\mathbf{x}_{ik}} f_{ij}(\cdot)}$, $k=1,\cdots, n_i,l=1,\cdots, p_{ij}$ are  available for agent $j$ in coalition $i$.
	
	2) The number of agents in coalition $i$, i.e.,  $n_i$, the average consensus gain in coalition $i$, i.e., $\gamma_i$ are known before playing the game for each agent in coalition $i$. The number $q_{ijl}, l=1,\cdots, p_{ij}$ is known for agent $j$ in coalition $i$. The constant $T_2$ and the consensus tracking gains $\alpha$ and $\beta$ are known for all agents in the game. 
	
	3) Each agent communicates $\mathbf{y}_{ij}^k, k=1,\cdots, n_i$ with its neighbors in $\mathcal{G}_i$, and communicates $\mathbf{z}_{ij, wq}, w\in\mathcal{N}, q\in\mathcal{N}_w, [w,q]\neq [i,j]$ with its neighbors in $\mathcal{G}$. 
	
	\subsubsection{Convergence Analysis}
	
	The main result of this section can be summarized as follows.
	
	\begin{theorem}
		Suppose that Assumptions  \ref{innercoal}-\ref{coali} hold. Let $T$ be an arbitrary positive constant. The parameters are selected such that $\alpha\geq 0$,  $\beta>c_2$, $\max_{w\in\mathcal{N}, q\in\mathcal{N}_w}\{ {\frac{\lambda_{max}(H_{wq})}{\lambda_{min}(H_{wq})}}\frac{1}{\beta-c_2}\sum_{i=1}^N\sum_{j=1}^{n_i}\Vert\mathbf{z}_{ij, wq}(0)-\mathbf{x}_{wq}(0)\Vert_\infty\}\leq T_2<T$,  $\gamma_i>c_{6}(T,\alpha,\beta)$, and  $T_5\triangleq\max_{i\in\mathcal{N}, k\in\mathcal{N}_i}\{ \frac{1}{2(\gamma_i-c_{6}(T,\alpha, \beta))}\sum_{j=1}^{n_i}\sum_{r=1}^{n_i} a_{i,jr} \Vert\mathbf{y}_{ij}^k(0)-\mathbf{y}_{ir}^k(0) \Vert_\infty\} <T$, where $c_2$ was defined in  (\ref{eqvi}), and $c_6(T,\alpha,\beta)$ is a positive constant dependent on $T$, $\alpha$, and $\beta$. Then,  Algorithm  \ref{a4} guarantees that each Filippov solution $\mathbf{x}(t)$  converges to the unique VE of the game. \label{theorem4}
	\end{theorem}
	
	\begin{proof}
		See Appendix \ref{ap7}
	\end{proof}
	
	\begin{remark} 
		The parameters satisfying Theorem \ref{theorem4} always exist. For an arbitrary $T$, we can always select a sufficiently large $\beta$ such that $T_2$ exists. Then, $\gamma_i$ can be determined according to $\beta$.
	\end{remark}

	\section{Numerical Simulation} \label{simu}

	In this section, we consider an example where three companies are competing for customers by deciding the amount of a commodity. Each company (coalition) has one or multiple subsidiary companies (agents) coordinating with each other on the amount of the commodity that each subsidiary company needs to produce, as shown in Fig. \ref{fig:coalitionexample}. The amount of the commodity that a subsidiary produces is denoted by $x_{ij}$, $i=1,2,3, j=1,\cdots, n_i$ with $n_1=1, n_2=3, n_3=6$. 
	
	\begin{figure} [htbp]
		\centering
		\includegraphics[width=0.8\linewidth]{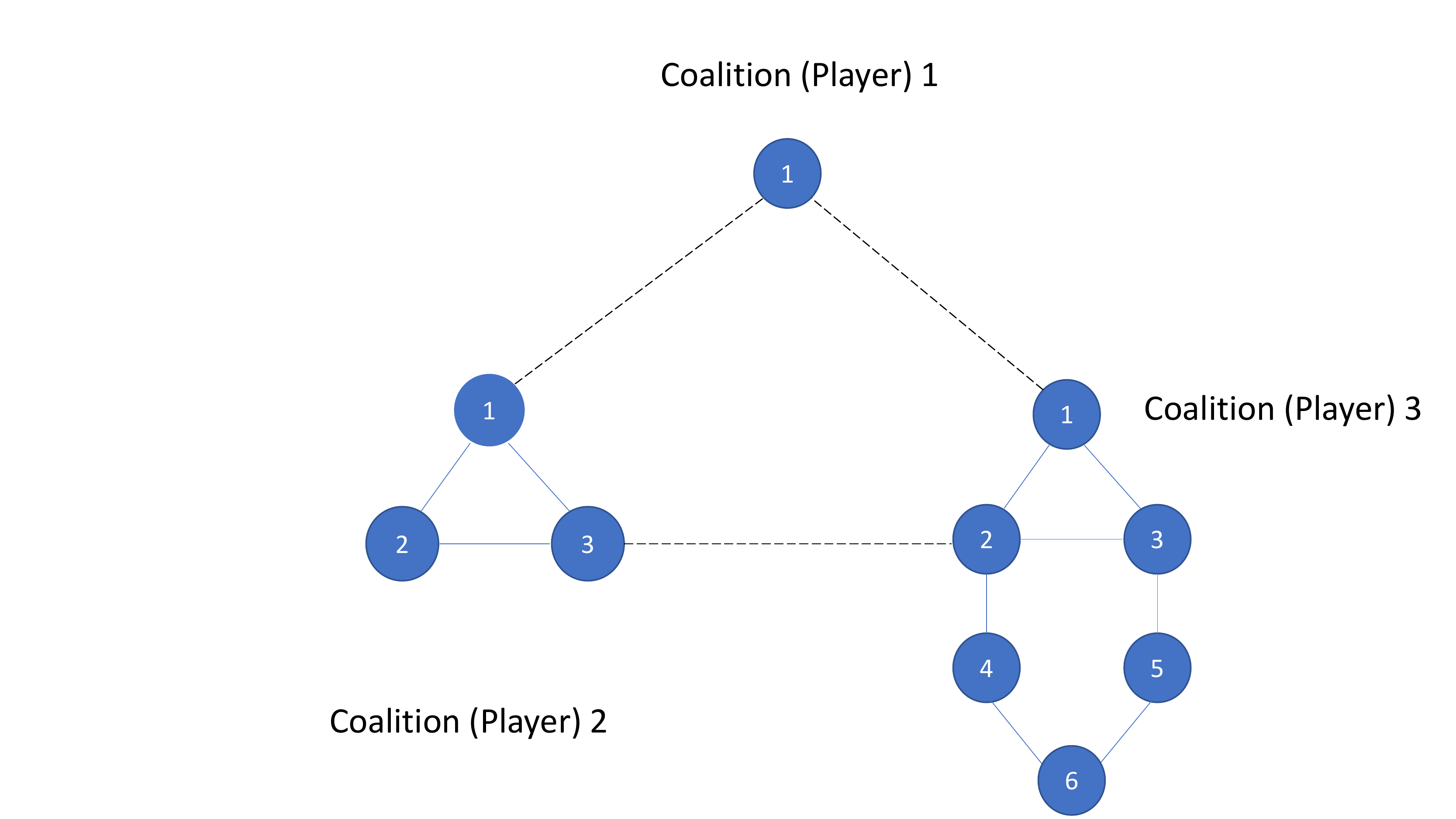}
		\caption{The graph topology.}
		\label{fig:coalitionexample}
	\end{figure}
	
	The local objective function of agent $j$ in coalition $i$ can be described as $f_{ij}(\mathbf{x})=c_{ij}(\mathbf{x}_{ij})-p_{ij}(\mathbf{x})\mathbf{x}_{ij}$, where $c_{ij}(\mathbf{x}_{ij})=s_{ij}\mathbf{x}_{ij}^2+r_{ij}\mathbf{x}_{ij}$ is the cost function,  $p_{ij}(\mathbf{x})=d_{ij}-\mathbf{u}_{ij}^T\mathbf{x}$ denotes the price function,  $s_{ij}, r_{ij}, d_{ij}\in\mathbb{R}$, and $\mathbf{u}_{ij}\in\mathbb{R}^{10}$. In the simulation, we let $d_{ij}=40$ and $\mathbf{u}_{ij}=0.2 * \mathbf{1}_{10}$.  The values of the other parameters are depicted in Table \ref{table1}. It can be verified that the game mapping is strongly monotone.
	
	Four shared constraints are considered in the simulation, which can be described as $g_s(\mathbf{x})=A_s\mathbf{x}-b_s$, where $A_1=-[1,\cdots,1]$, $A_2=[1,2,1,3,6,1,6,2,6,6]$, $A_3=[1,2,1.2,2,4,3,3,3,3,3]$, $A_4=[1,2,5,1,1,2,1,2,1,1]$, and  $b_1=-20,b_2=b_3=b_4=180$. In addition, each agent has a local constraint  set $\Omega_{ij}=\{\mathbf{x}_{ij}\in\mathbb{R}| 0 \leq \mathbf{x}_{ij}\leq 10 \}$. 
	
	It can be verified that the problem investigated in this section satisfying Assumptions  \ref{innercoal}-\ref{coali}.
	
	\begin{table}[htbp]
		\centering
		\caption{The values of the parameters.}
		\begin{tabular}{|c|c|c|c|c|}
			\hline
			(Coalition, Agent) & $s_{ij}$ & $r_{ij}$   \\
			\hline
			(1,1) & 3.2 & 4   \\
			(2,1) & 2.3 & 2   \\
			(2,2) & 2.3 & 6  \\
			(2,3) & 2.9 & 2   \\
			(3,1) & 1.5 & 1   \\
			(3,2) & 1.4 & 2   \\
			(3,3) & 2.1 & 5    \\
			(3,4) & 2.2 & 2    \\	
			(3,5) & 1.2 & 1   \\	
			(3,6) & 1.5 & 3   \\	
			\hline			
		\end{tabular}
		\label{table1}
	\end{table}
	
	The unique VE of the game is $[3.945, 5.388, 4.677, 4.003, 5.946, 8.138, 3.015, 4.629, 7.729$, $5.176]^T$. Let $\gamma_i=300$, $\alpha=1$, $\beta=100$, and $T_2=20$. Figs.	\ref{fig:coalition1and2}-\ref{fig:coalition31} show the convergence to the VE using Algorithm \ref{a3}.  It can be seen from the simulation results that the trajectories converge to the VE of the game, which verifies Theorem \ref{theorem3}.  Fig \ref{multi} shows the evolution of the multipliers of agent 1 in coalition 1, which are bounded. 
	
	Next, we consider the distributive constraint information scenario, where the constraint information distribution is shown in Fig.	\ref{fig:coalitionexample2}.  Each subsidiary company has access to only a subset of the constraints, based on the efficiency consideration.
	
	Let $\gamma_i=1000$, $\alpha=1$, $\beta=100$, and $T_2=20$. Figs.	\ref{fig:coalition1and2d}-\ref{fig:coalition31d} show the convergence to the VE using Algorithm \ref{a4}. 
	
	\begin{figure} [htbp]
		\centering
		\includegraphics[width=0.7\linewidth]{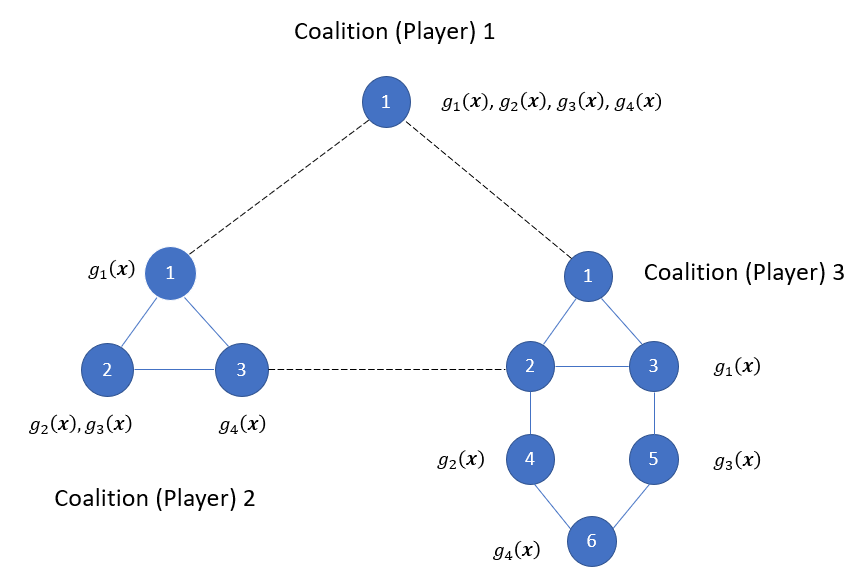}
		\caption{The distributive constraint information scenario.}
		\label{fig:coalitionexample2}
	\end{figure}

	\begin{figure}
		\centering
		\includegraphics[width=0.9\linewidth]{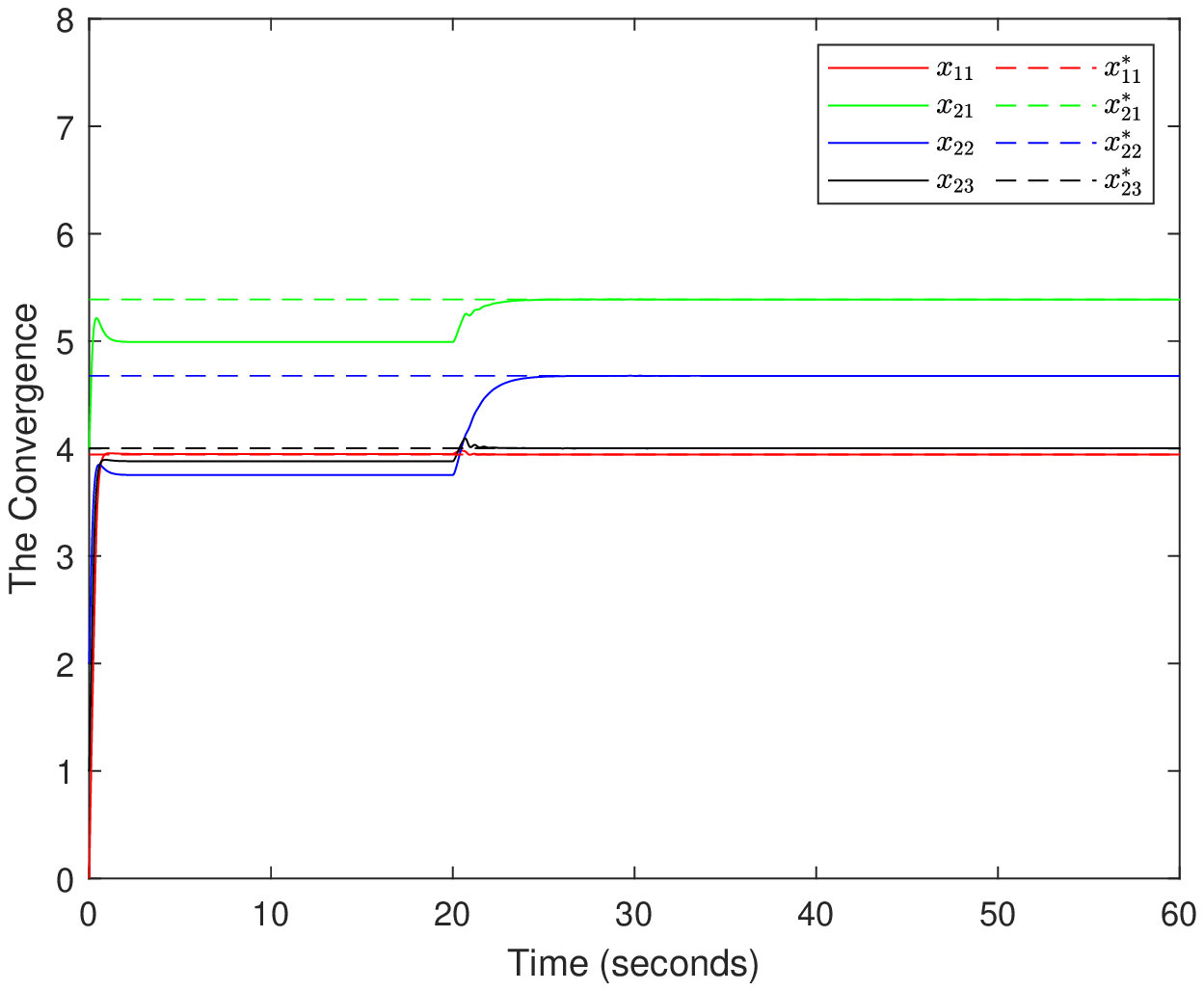}
		\caption{The convergence to the VE: agents in coalitions 1 and 2.}
		\label{fig:coalition1and2}
	\end{figure}
	\begin{figure}
		\centering
		\includegraphics[width=0.9\linewidth]{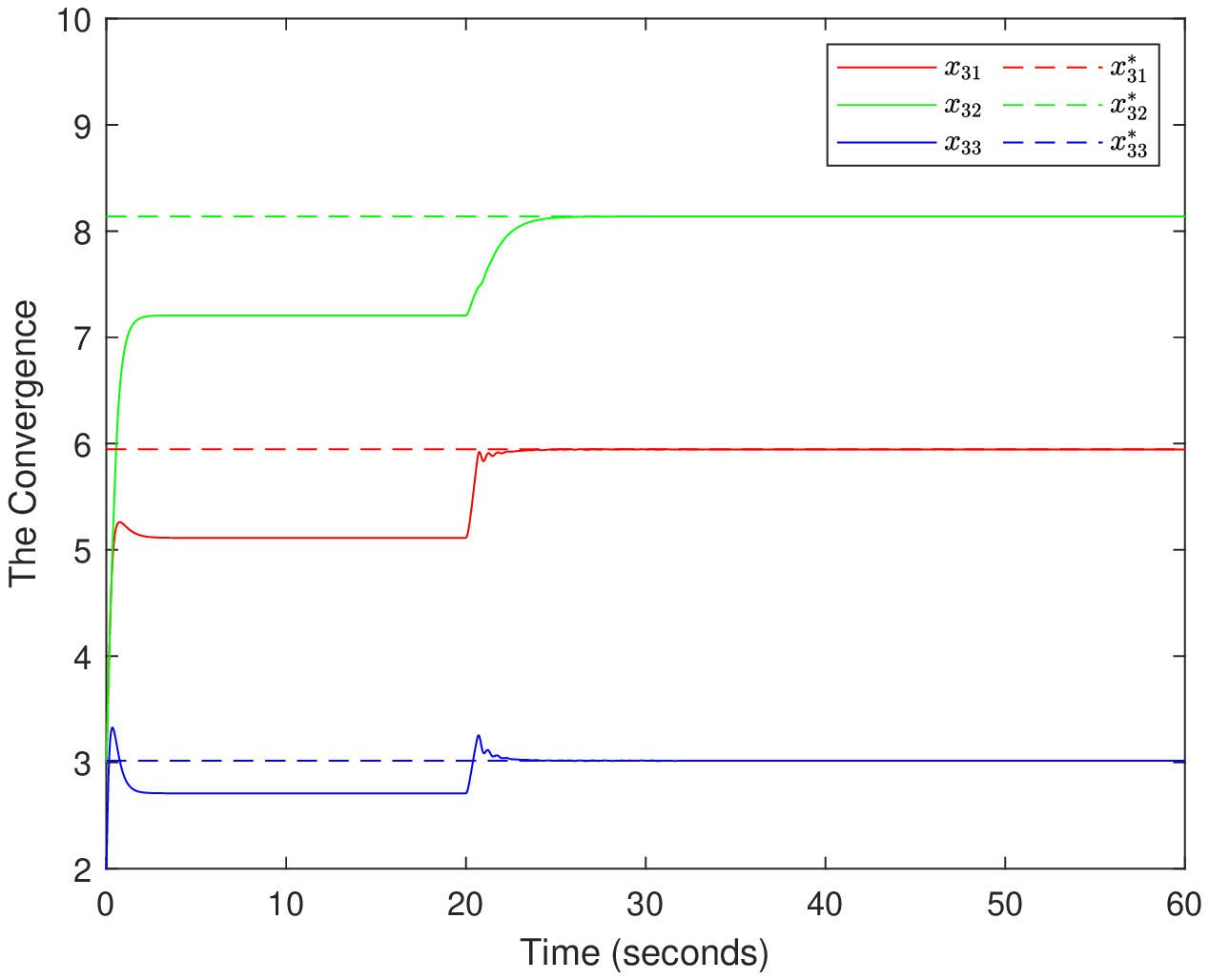}
		\caption{The convergence to the VE: agents $1$-$3$ in coalition 3.}
		\label{fig:coalition32}
	\end{figure}
	\begin{figure}
		\centering
		\includegraphics[width=0.9\linewidth]{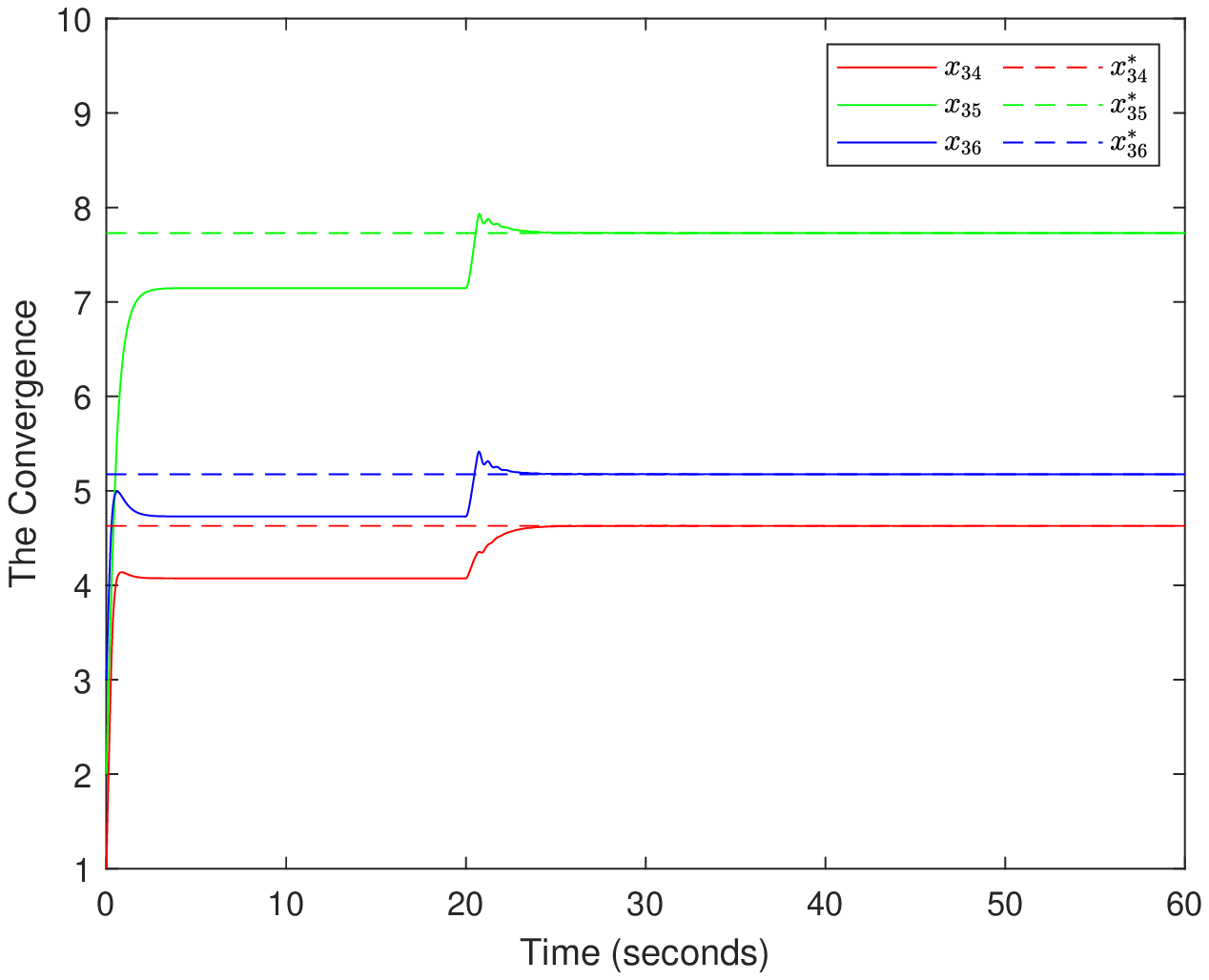}
		\caption{The convergence to the VE: agents $4$-$6$ in coalition 3.}
		\label{fig:coalition31}
	\end{figure}
	
	\begin{figure}
		\centering
		\includegraphics[width=0.9\linewidth]{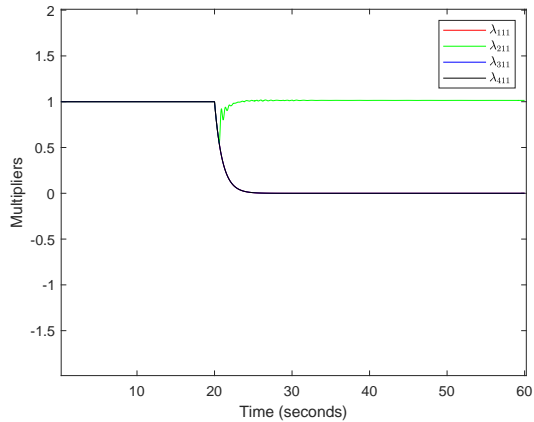}
		\caption{Multipliers: agent 1 in coalition 1.}
		\label{multi}
	\end{figure}

	\begin{figure}
		\centering
		\includegraphics[width=0.9\linewidth]{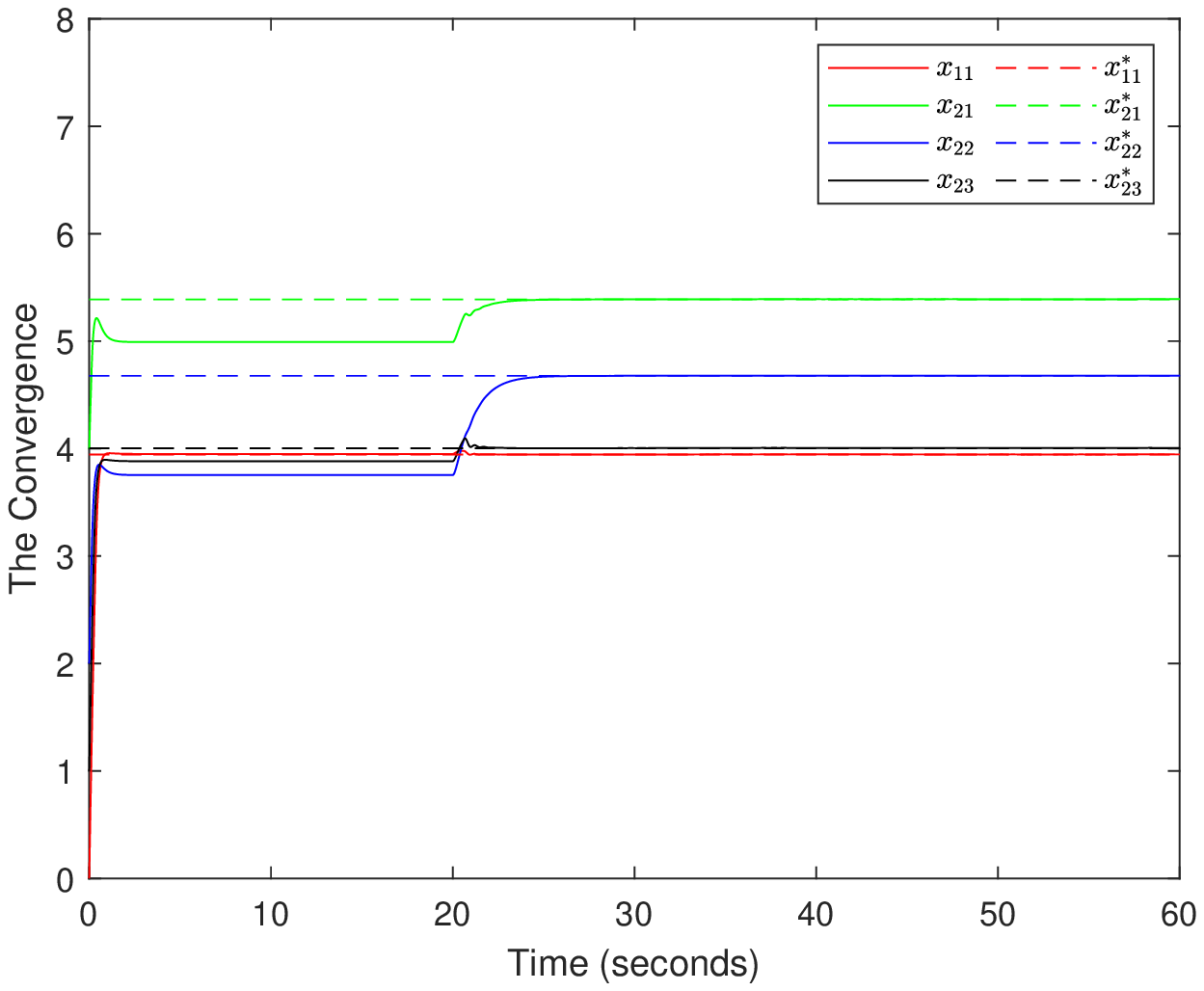}
		\caption{The convergence to the VE: agents in coalitions 1 and 2.}
		\label{fig:coalition1and2d}
	\end{figure}
	\begin{figure}
		\centering
		\includegraphics[width=0.9\linewidth]{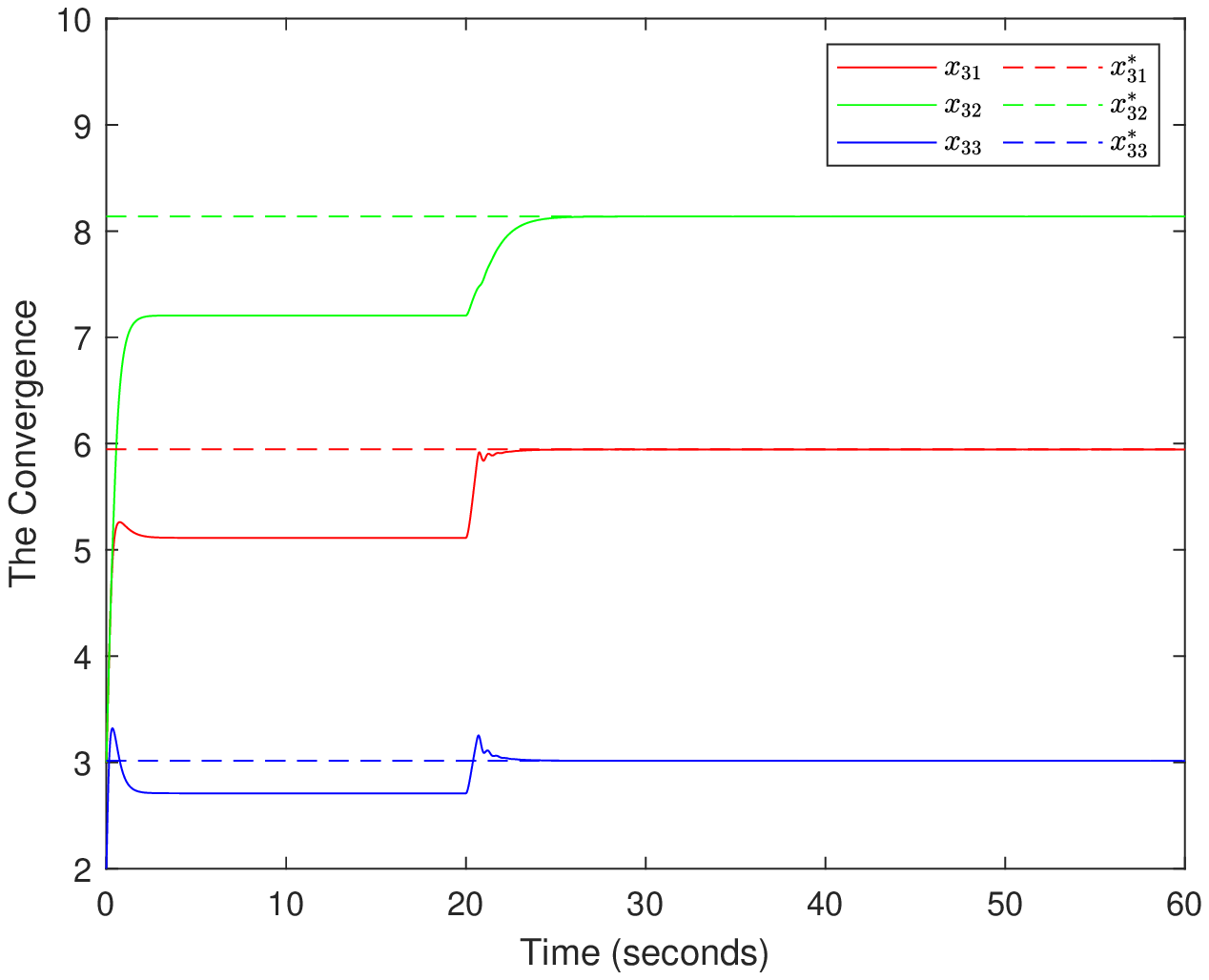}
		\caption{The convergence to the VE: agents $1$-$3$ in coalition 3.}
		\label{fig:coalition32d}
	\end{figure}
	\begin{figure}
		\centering
		\includegraphics[width=0.9\linewidth]{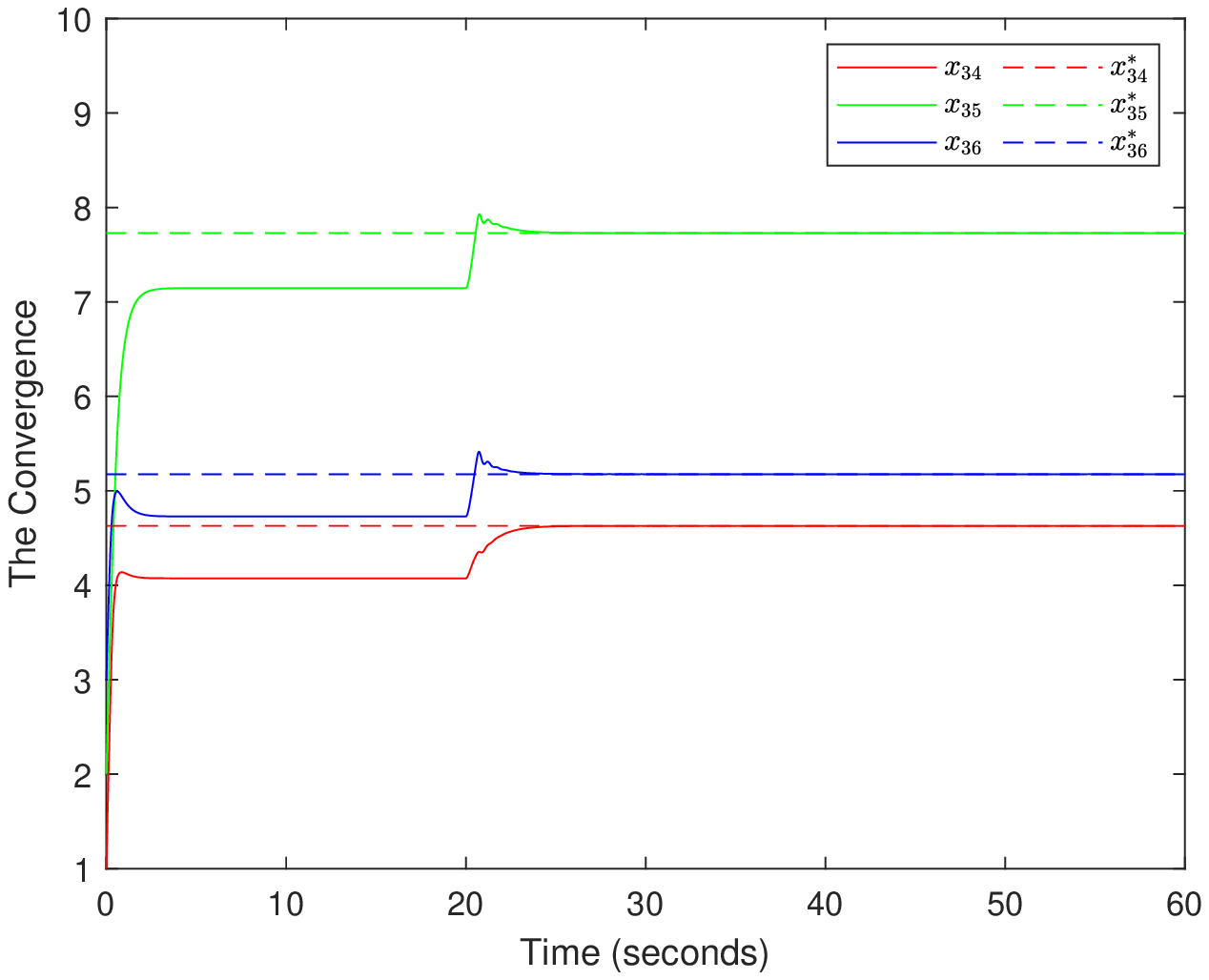}
		\caption{The convergence to the VE: agents $4$-$6$ in coalition 3.}
		\label{fig:coalition31d}
	\end{figure}

	\color{black}
	
	\section{Conclusions} \label{conclu}
	
	In this work, we proposed a finite-time consensus-based strategy for the GNE seeking problem of $N$-coalition games with inequality constraints. The main contribution of this work is the distributed algorithm design for $N$-coalition games with inequality constraints and a distributive constraint information-based algorithm. Full decision information algorithms and partial decision information algorithms were proposed. Lyapunov analysis was conducted to prove the convergence to the VE of the game.
	
	\bibliographystyle{IEEEtran}
	\bibliography{bib}
	
	\appendix
	
	\subsection{Proof of Lemma \ref{lemma2}} \label{ap1}
	
	According to Page 94 of \cite{facchinei2007finite}, $\mathbf{x}^*\in\Delta$ is a solution of VI$(F(\mathbf{x}), \Delta)$ if and only if  $0\in F(\mathbf{x}^*)+\mathcal{N}_{cone}(\mathbf{x}^*, \Delta)$, which is equivalent to $0\in  F(\mathbf{x}^*)+[\mathcal{T}(\mathbf{x}^*, \Delta)]^\circ$ since $\Delta$ is convex and closed (Corollary 5.2.5 of \cite{hiriart2013convex}), where $\mathcal{T}(\mathbf{x}^*, \Delta)$ is the tangent cone to  $\Delta$ at $\mathbf{x}^*$, and $[\cdot]^\circ$ represents the polar cone of a cone.
	
	Based on Theorem 3.20 of \cite{ruszczynski2011nonlinear}, under the Slater's condition, $[\mathcal{T}(\mathbf{x}^*, \Delta)]^\circ=[\mathcal{T}(\mathbf{x}^*, \Omega)]^\circ+\sum_{s\in I^\circ(\mathbf{x}^*)}\text{cone}(\{\nabla_\mathbf{x} {g_s}(\mathbf{x}^*)\})$, where $\text{cone}(X)$ represents a convex cone defined by $\text{cone}(X)=\{\sigma z| \sigma>0, z\in  X \}$, and $I^\circ(\mathbf{x}^*)=\{s|s\in\{1,\cdots, p\}, g_s(\mathbf{x}^*)=0\}$ represents the set of the subscripts of the active constraints (i.e., $g_s(\mathbf{x}^*)=0$, $s\in\{1,\cdots, p\}$).
	
	Then, since $[\mathcal{T}(\mathbf{x}^*, \Omega)]^\circ=\mathcal{N}_{cone}(\mathbf{x}^*, \Omega)$, $\mathbf{x}^*\in \Delta$\\
	is a solution of VI$(F(\mathbf{x}), \Delta)$ if and only if   there exists constants $\sigma_s^*>0, s\in I^\circ(\mathbf{x}^*)$ such that 
	\begin{eqnarray}
		0\in F(\mathbf{x}^*)+\sum_{s\in I^\circ(\mathbf{x}^*)}\sigma_{s}^*\nabla_{\mathbf{x}}g_s(\mathbf{x}^*)+\mathcal{N}_{cone}(\mathbf{x}^*, \Omega). \label{KKT3}
	\end{eqnarray}
	
	According to the definition of active constraints, the statement that there exists constants $\sigma_s^*>0, s\in I^\circ(\mathbf{x}^*)$ such that (\ref{KKT3}) holds is equivalent to the statement that there exists non-negative Lagrangian multipliers $\lambda^*_s, s=1,\cdots, p$  such that (\ref{kkt2}) holds. The proof is completed.
	
	\subsection{Proof of Lemma \ref{lemma3}} \label{ap2}
	
	\textit{Proof of 1)}:  According to Theorem 3.27 of \cite{ruszczynski2011nonlinear}, for each player's optimization problem, the KKT condition (\ref{kkt}) is sufficient to guarantee that $\mathbf{x}^*_{i}$ is an optimal solution to the optimization problem of $f_{i}(\mathbf{x}_i,\mathbf{x}^*_{-i})$,  subject to $\mathbf{x}_{i}\in\Omega_i\cap\{{\mathbf{x}_i\in\mathbb{R}^{m_i}|} g(\mathbf{x}_i,\mathbf{x}^*_{-i})\leq 0\}, \forall i \in \mathcal{N}.$ According to Definition \ref{ref}, $\mathbf{x}^*$ is a GNE.
	
	\textit{Proof of 2)}: It is a direct result of Lemma \ref{lemma2}.
	
	%

	%
	%
	
	\subsection{Proof of Lemma \ref{lemma1}} \label{ap3}
	The proof is similar to Theorem 3.1 of \cite{yi2016initialization}, which uses the fact that for any $z\in\Omega_i$, $\Omega_i\subset z+\mathcal{T}(z, \Omega_i)$  (see Page 219 of \cite{aubin1984differential}), and the statements on  Page 174 and Page 214 of \cite{aubin1984differential}. A similar conclusion also appears in \cite{liao2004neurodynamical}.
	
	\subsection{Proof of Theorem \ref{theorem1}} \label{ap4}

	Consider the dynamics in (\ref{al5a}) and (\ref{al5e}). For all $t\geq 0$, we have ${\mathbf{x}}(t)\in\Omega$ according to Lemma \ref{lemma1}. Furthermore, based on (\ref{al5e}), since $\mathcal{P}_{\mathbb{R}_{\geq 0}}[\lambda_{s}+g_s(\mathbf{x})]-\lambda_{s}\leq |g_s(\mathbf{x})|$ for $\lambda_s>0$, we have  ${\lambda}_s(t)\leq c_at+\lambda_s(0)$, where $c_a\geq |g_{s}(\mathbf{x})|$ for all $s\in\{1\cdots, p\}$ and $t\geq 0$, while the existence of $c_a$ can be guaranteed  by  Lemma \ref{lemma1}. 
	
	Motivated by \cite{yi2016initialization}, let $\theta=[\mathbf{x}^T,\lambda^T]^T$ with $\lambda=[\lambda_s]_{s\in\{1,\cdots,p\}}$, $U(\theta)=[(F(\mathbf{x})+\sum_{s=1}^p\lambda_sG_s(\mathbf{x}))^T,-g(\mathbf{x})^T]^T$, $\Sigma=\Omega\times \mathbb{R}_{\geq 0}^p$, and $H(\theta)=P_\Sigma(\theta-U(\theta))$.
	
	Define a Lyapunov candidate function $V(\theta)=(\theta-H(\theta))^TU(\theta)-\frac{1}{2}\Vert\theta-H(\theta)\Vert^2+\frac{1}{2}\Vert\theta-\theta^*\Vert^2$, where $\theta^*=[\mathbf{x}^{*T},\lambda^{*T}]^T$, $\mathbf{x}^*$ is the VE,  $ \lambda^*=[\lambda_s^*]_{s\in\{1,\cdots,p\}}$ and $\lambda_s^*$,  $s=1,\cdots, p$, are multipliers that satisfy (\ref{kkt}). Then, based on \cite{fukushima1992equivalent}, $ V(\theta)\geq \frac{1}{2}\Vert\theta-\theta^*\Vert^2$.
	
	According to Lemma \ref{lemma1} and the fact that ${\lambda}_s(t)\leq  c_at+\lambda_s(0)$, there exist  positive constants $c_{b}$ and $c_c$ such that for all $t\geq 0$, $V(t)\leq c_b+c_ct^2$. Thus, $V(T_1)\leq c_b+c_cT_1^2$, which is bounded.
	
	For $t\geq T_1$, according to Lemma \ref{l6}, $\mathbf{y}_{ij}^k(t)=\sum_{r=1}^{n_i}{\nabla_{\mathbf{x}_{ik}} f_{ir}(\mathbf{x})}$.
	
	Then, (\ref{al5a}) and  (\ref{al5e}) satisfy that for $t\geq T_1$,
	\begin{align}
		\dot{\theta}=H(\theta)-\theta. \label{V}
	\end{align}
	
	The stability analysis for (\ref{V}) is encouraged by Theorem 3 of \cite{liang2017distributed} with the constraints becoming nonlinear.
	
	According to Assumptions \ref{fassumption} and \ref{assumption_convex}, $U(\theta)$ is continuously differentiable. According to Lemma 3 of \cite{liang2017distributed} (Theorem 3.2 of \cite{fukushima1992equivalent}), $\nabla V=U(\theta)+(\mathcal{J}U(\theta)-I)(\theta-H(\theta))+(\theta-\theta^*)$, where $\mathcal{J}U(\theta)$ is the Jacobian
	matrix of $U(\theta)$. Taking the derivative of $V$ along (\ref{V}) for $t\geq T_1$  gives 
	\begin{align}
		\dot{V}&=(\nabla V)^T\dot{\theta} \notag \\
		&=(\nabla V)^T(H(\theta)-\theta) \notag \\
		&=-W_1(\theta)-W_2(\theta)-W_3(\theta)-W_4(\theta),
	\end{align}
	where
	\begin{align}
		W_1(\theta)&=(\theta^*-H(\theta))^T(U(\theta)+H(\theta)-\theta), \notag \\
		W_2(\theta)&=(\theta-\theta^*)^TU(\theta^*), \notag \\
		W_3(\theta)&=(\theta-\theta^*)^T(U(\theta)-U(\theta^*)), \notag \\
		W_4(\theta)&=(H(\theta)-\theta)^T\mathcal{J}U(\theta)(H(\theta)-\theta).
	\end{align}

	According to Lemma \ref{projector}-2) ($y=\theta^*, P_{\Omega}[x]= H(\theta), x=\theta-U(\theta)$),  one can get that $W_1(\theta)\geq 0$. 
	
	For any $\theta_a=[\mathbf{x}_a^T,\lambda_a^T]^T$ and $\theta_b=[\mathbf{x}_b^T,\lambda_b^T]^T$ with $\mathbf{x}_a, \mathbf{x}_b\in\mathbb{R}^M$ and $\lambda_a=[\lambda_{as}]_{s\in\{1,\cdots, p\}}, \lambda_b=[\lambda_{bs}]_{s\in\{1,\cdots, p\}} \in\mathbb{R}_{\geq 0}^p$, we have
	\begin{align}
		&(\theta_a-\theta_b)^T(U(\theta_a)-U(\theta_b))\notag \\=&(\mathbf{x}_a-\mathbf{x}_b)^T(F(\mathbf{x}_a)-F(\mathbf{x}_b))\notag\\
		+&(\mathbf{x}_a-\mathbf{x}_b)^T(\sum_{s=1}^p\lambda_{as}G_s(\mathbf{x}_a)-\sum_{s=1}^p\lambda_{bs}G_s(\mathbf{x}_b))\notag\\
		-&(\lambda_a-\lambda_b)^T(g(\mathbf{x}_a)-g(\mathbf{x}_b))\geq 0,
	\end{align}
	where the strict monotonicity of $F(\mathbf{x})$ and the convexity of $g(\mathbf{x})$ are used. Thus, $W_3(\theta)\geq 0$. Furthermore,  $U(\theta)$ is monotone. According to Proposition 2.3.2 of \cite{facchinei2007finite}, $\mathcal{J}U(\theta)$ is positive semidefinite,  which implies that $W_4(\theta)\geq 0$.

	Based on the KKT condition (\ref{kkt2}), Lemma \ref{lemma1}, the definition of normal cone, $(\mathbf{x}-\mathbf{x}^*)^T(F(\mathbf{x}^*)+\sum_{s=1}^p\lambda_sG_s(\mathbf{x}^*))\geq 0$. Based on Lemma  \ref{lemma1_1}, $-(\lambda-\lambda^*)^Tg(\mathbf{x}^*)=-\lambda^Tg(\mathbf{x}^*)\geq 0$. It follows that $W_2(\theta)\geq 0$.

	Then, $\dot{V}\leq 0$ for all $t\geq T_1$. It follows that $0 \leq V(t)\leq V(T_1)$ for all $t\geq T_1$, which implies the boundedness of $\lambda$. Furthermore, since $\dot{V}\leq -W_3(\theta)\leq -(\mathbf{x}-\mathbf{x}^*)^T(F(\mathbf{x})-F(\mathbf{x}^*))\leq 0$, we have $\int_{T_1}^t(\mathbf{x}-\mathbf{x}^*)^T(F(\mathbf{x})-F(\mathbf{x}^*))dt\leq V(T_1)$. According to the Monotone Convergence Theorem, $\lim_{t\rightarrow \infty}\int_{T_1}^t(\mathbf{x}-\mathbf{x}^*)^T(F(\mathbf{x})-F(\mathbf{x}^*))dt$ exists and is finite. According to the uniform continuity of $(\mathbf{x}-\mathbf{x}^*)^T(F(\mathbf{x})-F(\mathbf{x}^*))$ and the Barbalat's Lemma, $(\mathbf{x}-\mathbf{x}^*)^T(F(\mathbf{x})-F(\mathbf{x}^*))\rightarrow 0$ as $t\rightarrow \infty$. Let $\mathbf{x}^k$ be an arbitrary  subsequence of $\mathbf{x}$. According to the boundedness of $\mathbf{x}^k$, there exists at least one accumulation point. Let $\bar{x}$ be an accumulation point, i.e., there exists a convergent subsequence $\mathbf{x}^{k_s}$ of $\mathbf{x}^k$ such that $\lim_{k_s\rightarrow\infty}\mathbf{x}^{k_s}=\bar{x}$. According to the continuity of $F(\mathbf{x})$ and the composite function limit, $\lim_{k_s\rightarrow\infty}(\mathbf{x}^{k_s}-\mathbf{x}^*)^T(F(\mathbf{x}^{k_s})-F(\mathbf{x}^*))=(\bar{x}-\mathbf{x}^*)^T(F(\bar{x})-F(\mathbf{x}^*))$, which equals to zero, since any subsequence of  $(\mathbf{x}-\mathbf{x}^*)^T(F(\mathbf{x})-F(\mathbf{x}^*))$ tends to zero. According to the strict monotonicity of  $F(\mathbf{x})$, $\bar{x}=\mathbf{x}^*$. Then, $\mathbf{x}^k$ has a unique accumulation point $\mathbf{x}^*$, which implies that $\lim_{k\rightarrow\infty}\mathbf{x}^k=\mathbf{x}^*$. Since $\mathbf{x}^k$ is arbitrary, $\lim_{t\rightarrow\infty} \mathbf{x}=\mathbf{x}^*$.
	
	\subsection{Proof of Theorem \ref{theorem3}} \label{ap6}
	
	According to Lemma \ref{l3}, we can get the following conclusion for the finite-time consensus algorithm in (\ref{finite}).
	
	\begin{lemma}
		If $\beta>c_2$, then for all $t\geq T_2$, where $T_2$ is a constant such that $T_2\geq \max_{w\in\mathcal{N}, q\in\mathcal{N}_w}\{ {\frac{\lambda_{max}(H_{wq})}{\lambda_{min}(H_{wq})}}\frac{1}{\beta-c_2}\sum_{i=1}^N\sum_{j=1}^{n_i}\Vert\mathbf{z}_{ij, wq}(0)-\mathbf{x}_{wq}(0)\Vert_\infty\}$, we have $\mathbf{z}_{ij, wq}(t)=\mathbf{x}_{wq}(t)$. Moreover, $\Vert\mathbf{z}_{ij, wq}(t)-\mathbf{x}_{wq}(t)\Vert_\infty^2\leq  { \frac{\lambda_{\max}(H_{wq})}{\lambda_{\min}(H_{wq})}}\sum_{i=1}^N\sum_{j=1}^{n_i}\Vert\mathbf{z}_{ij, wq}(0)-\mathbf{x}_{wq}(0)\Vert_\infty^2$ for all $t\geq 0$. \label{fc}
	\end{lemma}
	
	Thus, according to the boundedness of $\mathbf{z}_{ij}$ in Lemma \ref{fc} and $\dot{\mathbf{z}}_{ij}$ (based on (\ref{finite})), there exists a constant $c_4$ such that
	\begin{align}
		\left\Vert
		\frac{d(n_i{\nabla_{\mathbf{x}_{ik}} f_{ij}(\mathbf{z}_{ij})})}{dt}\right\Vert_\infty\leq c_4, \label{d1}
	\end{align}
	where $c_4$  depends on the bound of the set $\Omega$, $n_i$,  the bound of some second-order Hessian entries in a compact set, graph information, $\alpha$, $\beta$, and  initial values.
	
	Then, based on the finite-time consensus result in Lemma \ref{fc}, we can get the following conclusion for the finite-time average consensus tracking algorithm in (\ref{bl3b}) and  (\ref{bl3c}), similar to the result in  Lemma \ref{l6}.
	
	\begin{lemma}
		If $\gamma_i>c_4$, then for all $t\geq T_4\triangleq \max\{T_2, T_3\}$ where $ T_3\triangleq\max_{i\in\mathcal{N}, k\in\mathcal{N}_i}\{ \frac{1}{2(\gamma_i-c_4)}\sum_{j=1}^{n_i}\sum_{r=1}^{n_i}a_{i,jr}\Vert\mathbf{y}_{ij}^k(0)-\mathbf{y}_{ir}^k(0) \Vert_\infty\}$, we have	$\mathbf{y}_{ij}^k(t)=\sum_{r=1}^{n_i}{\nabla_{\mathbf{x}_{ik}} f_{ir}(\mathbf{x})}$. \label{l8}
	\end{lemma}
	
	The following proof is similar to that of Theorem \ref{theorem1}. Before  $T_4$, all the state variables remain bounded. After $T_4$, the system is equivalent to the system described in (\ref{V}). Thus, we can prove that $\lim_{t\rightarrow\infty} \mathbf{x}=\mathbf{x}^*$.

	\subsection{Proof of Theorem \ref{t2}} \label{ap5}
	
	For all $t\geq 0$, according to the analysis in Theorem \ref{theorem1}, 	${\lambda}_s(t)\leq{\lambda}_s(0)+\int_{\tau=0}^tc_ad\tau={\lambda}_s(0)+ c_at$. Consider the same Lyapunov function as the previous section, i.e.,  $V(\theta)=(\theta-H(\theta))^TU(\theta)-\frac{1}{2}\Vert\theta-H(\theta)\Vert^2+\frac{1}{2}\Vert\theta-\theta^*\Vert^2.$ Then, $V(t)\leq c_b+c_ct^2$  for some $c_b$ and $c_c$.
	
	Let $T>0$ be an arbitrary and fixed constant. For any $\theta$ such that $V(\theta)\leq c_b+c_cT^2$, we have $\frac{1}{2}\Vert\theta-\theta^*\Vert^2\leq V(\theta)\leq c_b+c_cT^2$. Thus, there exists a positive constant   $c_5(T)$ such that $\Vert\frac{d}{dt}(n_i{\nabla_{\mathbf{x}_{ik}} f_{ij}(\mathbf{x})}+\sum_{s=1}^p\delta_{ijs}\frac{n_i}{q_{is}}\lambda_s\nabla_{\mathbf{x}_{ik}} g_{s} (\mathbf{x}))\Vert_\infty< c_5(T)$ for all $i\in\mathcal{V}, j,k \in\mathcal{N}_i$ and all $\theta$ satisfying $V(\theta)\leq c_b+c_cT^2$. 
	
	Denote $T_3=\max_{{i\in\mathcal{N}, k\in\mathcal{N}_i}}\{ \frac{1}{2(\gamma_i-c_5(T))}\sum_{j=1}^{n_i}\sum_{r=1}^{n_i}a_{i,jr}$ $\Vert\mathbf{y}_{ij}^k(0)-\mathbf{y}_{ir}^k(0) \Vert_\infty\}$. Then,  we have the following conclusion.
	
	\begin{lemma}
		For any $T>0$, the dynamics in (\ref{al7b1}) and  (\ref{al7c1}) guarantee that if $\gamma_i>c_5(T)$ and $T_3<T$, then 1) for all $0 \leq t\leq T$,  	$\Vert\frac{d}{dt}(n_i{\nabla_{\mathbf{x}_{ik}} f_{ij}(\mathbf{x})}+\sum_{s=1}^p\frac{n_i}{q_{is}}\delta_{ijs}\lambda_s\nabla_{\mathbf{x}_{ik}} g_{s} (\mathbf{x}))\Vert_\infty
		< c_5(T)$ for all $i\in\mathcal{V}, j,k \in\mathcal{N}_i$; 2) for all $T_3\leq t\leq T$, 	$\mathbf{y}_{ij}^k(t)=\sum_{r=1}^{n_i}{\nabla_{\mathbf{x}_{ik}} f_{ir}(\mathbf{x})}+\sum_{s=1}^p\lambda_s\nabla_{\mathbf{x}_{ik}} g_{s} (\mathbf{x})$. \label{lem}
	\end{lemma}
	
	\textit{Proof of Lemma \ref{lem}: Since $V(t)\leq c_b+c_ct^2$ for all $t\geq 0$, we have $V(t)\leq c_b+c_cT^2$ for $0 \leq t\leq T$. Then, 1) is proven according to the definition of $c_5(T)$ above this lemma. Based on the boundedness of the average consensus tracked variable within $[0,T]$, we can  utilize  Lemma \ref{l2} to prove the finite-time average consensus at $t=T_3$. Then, 2) is proven.}

	Thus, for all $T_2\leq t\leq T$, by a similar analysis as in Theorem \ref{theorem1}, we have $\dot{V}\leq 0$, and $V(t)\leq V(T_2)\leq c_b+c_cT_2^2<c_b+c_cT^2$. 
	
	Next, we can prove by contradiction that $\Vert\frac{d}{dt}(n_i{\nabla_{\mathbf{x}_{ik}} f_{ij}(\mathbf{x})}+\sum_{s=1}^p\frac{n_i}{q_{is}}\delta_{ijs}\lambda_s\nabla_{\mathbf{x}_{ik}} g_{s} (\mathbf{x}))\Vert_\infty
	<c_5(T)$ for all $t\geq T$ and $i\in\mathcal{V}, j,k \in\mathcal{N}_i$.
	
	Suppose that there exists a time instant $T'>T$ such that $\Vert\frac{d}{dt}(n_i{\nabla_{\mathbf{x}_{ik}} f_{ij}(\mathbf{x})}+\sum_{s=1}^p\frac{n_i}{q_{is}}\delta_{ijs}\lambda_s\nabla_{\mathbf{x}_{ik}} g_{s} (\mathbf{x}))\Vert_\infty
	< c_5(T)$ for all $T\leq t<T'$ and $i\in\mathcal{V}, j,k \in\mathcal{N}_i$, and $\Vert\frac{d}{dt}(n_i{\nabla_{\mathbf{x}_{ik}} f_{ij}(\mathbf{x})}+\sum_{s=1}^p\frac{n_i}{q_{is}}\delta_{ijs}\lambda_s\nabla_{\mathbf{x}_{ik}} g_{s} (\mathbf{x}))\Vert_\infty
	=c_5(T)$ for $t=T_3$ and some $i\in\mathcal{V}, j,k \in\mathcal{N}_i$.  According to the definition of $c_5(T)$, we have $V(T')>c_b+c_cT^2$.

	Furthermore, the following conclusions hold: i) for all $0 \leq t\leq T'$ and $i\in\mathcal{V}, j,k \in\mathcal{N}_i$,  $\Vert\frac{d}{dt}(n_i{\nabla_{\mathbf{x}_{ik}} f_{ij}(\mathbf{x})}+\sum_{s=1}^p\frac{n_i}{q_{is}}\delta_{ijs}\lambda_s\nabla_{\mathbf{x}_{ik}} g_{s} (\mathbf{x}))\Vert_\infty
	\leq c_5(T)$; ii) for all $T_2\leq t\leq T'$, 	$\mathbf{y}_{ij}^k(t)=\sum_{r=1}^{n_i}{\nabla_{\mathbf{x}_{ik}} f_{ir}(\mathbf{x})}+\sum_{s=1}^p\lambda_s\nabla_{\mathbf{x}_{ik}} g_{s} (\mathbf{x})$. Here, i) is obtained by assumption and ii) is obtained by utilizing Lemma 2.
	
	Then, for all $T_2\leq t\leq T'$,  we have $\dot{V}\leq 0$. Furthermore, $V(T')\leq V(T_2)\leq c_b+c_cT_2^2<c_b+c_cT^2$, which contradicts to the previous conclusion that $V(T')>c_b+c_cT^2$.
	
	Thus, for all $t\geq 0$, $\Vert\frac{d}{dt}(n_i{\nabla_{\mathbf{x}_{ik}} f_{ij}(\mathbf{x})}+\sum_{s=1}^p\frac{n_i}{q_{is}}\delta_{ijs}\lambda_s\nabla_{\mathbf{x}_{ik}} g_{s} (\mathbf{x}))\Vert_\infty
	<c_5(T)$. According to  Lemma \ref{l2},  for all $t\geq T_2$, 	$\mathbf{y}_{ij}^k(t)=\sum_{r=1}^{n_i}{\nabla_{\mathbf{x}_{ik}} f_{ir}(\mathbf{x})}+\sum_{s=1}^p\lambda_s\nabla_{\mathbf{x}_{ik}} g_{s} (\mathbf{x})$.
	
	The following analysis is similar to Theorem  \ref{theorem1}, and thus is omitted.

	\subsection{Proof of Theorem \ref{theorem4}} \label{ap7}
	
	Based on Lemma \ref{fc}, which is still true, the trajectory of $\mathbf{x}_{ij}$ is equivalent to the trajectory  in the following auxiliary system:
	\begin{subequations} \label{al81}
		\begin{align}
			\dot{\mathbf{x}}_{ij}&=\mathcal{P}_{\Omega_{ij}}[{\mathbf{x}}_{ij}-\mathbf{y}_{ij}^j]-{\mathbf{x}}_{ij}, \label{bl3a11}\\
			\mathbf{y}_{ij}^k&=\mathbf{w}_{ij}^k+n_i{\nabla_{\mathbf{x}_{ik}} f_{ij}(\mathbf{z}_{ij})}+\notag \\
			&\sum_{s=1}^p\frac{n_i}{q_{is}}\delta_{ijs}\lambda_s\nabla_{\mathbf{x}_{ik}} g_{s}(\mathbf{z}_{ij}), \label{bl3b11}\\
			\dot{\mathbf{w}}_{ij}^k&=-\gamma_i\sum_{r=1}^{n_i} a_{i,jr} \text{sgn}(\mathbf{y}_{ij}^k-\mathbf{y}_{ir}^k), \label{bl3c11}
		\end{align}
	\end{subequations}
	where $k=1,\cdots, n_i$, and for $s=1,\cdots, p$, 
	\begin{align}
		\dot{\lambda}_{s}=\begin{cases} 
			0, & t<T_2, \\
			\mathcal{P}_{\mathbb{R}_{\geq 0}}[\lambda_{s}+g_{s}(\mathbf{x})]-\lambda_{s}, & t\geq T_2,  
		\end{cases}  
		\label{bl3d11}
	\end{align}
	and
	\begin{align} 	\dot{\mathbf{z}}_{ij,wq}&=-\alpha\sum_{b\in\mathcal{N},h\in\mathcal{N}_b}a_{ij,bh}(\mathbf{z}_{ij, wq}-\mathbf{z}_{bh,wq}) \notag \\
		-\beta\text{sgn}&(\sum_{b\in\mathcal{N},h\in\mathcal{N}_b}a_{ij,bh}(\mathbf{z}_{ij, wq}-\mathbf{z}_{bh,wq})) \notag \\
		\text { for} & \text{ all }  w\in\mathcal{N}, q\in\mathcal{N}_w \text{ and }  [w,q]\neq [i,j]. \notag
	\end{align}
	
	Lemma \ref{l8} does not hold in this section.  	According to (\ref{bl3d11}), for all $t\geq 0$,	${\lambda}_s(t)\leq {\lambda}_s(0)+ c_at$, and there exist constants $c_b$ and $c_c$ such that the Lyapunov function	$V(\theta)=(\theta-H(\theta))^TU(\theta)-\frac{1}{2}\Vert\theta-H(\theta)\Vert^2+\frac{1}{2}\Vert\theta-\theta^*\Vert^2\leq c_b+c_ct^2$.
	
	Let $T>0$ be an arbitrary time constant, and $c_6(T,\alpha,\beta)$ be a constant  satisfying that  	$\Vert\frac{d}{dt}(n_i{\nabla_{\mathbf{x}_{ik}} f_{ij}(\mathbf{z}_{ij})}+ \sum_{s=1}^p\frac{n_i}{q_{is}}\delta_{ijs}\lambda_s\nabla_{\mathbf{x}_{ik}} g_{s} (\mathbf{z}_{ij}))\Vert_\infty< c_{6}(T,\alpha,\beta)$ for all $\theta$ satisfying $V(\theta)\leq c_b+c_bT^2$, and all $i\in\mathcal{N}$, $j,k\in\mathcal{N}_i$, where the boundedness of $\mathbf{z}_{ij}$ (see Lemma \ref{fc}) and $\dot{\mathbf{z}}_{ij}$ is used.
	
	Denote $T_5\triangleq\max_{i\in\mathcal{N}, k\in\mathcal{N}_i}\{ \frac{1}{2(\gamma_i-c_{6}(T,\alpha, \beta))}\sum_{j=1}^{n_i}\sum_{r=1}^{n_i}$ $a_{i,jr} \Vert\mathbf{y}_{ij}^k(0)-\mathbf{y}_{ir}^k(0) \Vert_\infty\}$. By selecting a sufficiently large  $\gamma_i$  such that   $T_5< T$, the following lemma can be obtained to replace Lemma \ref{l8}.
	
	\begin{lemma}
		Let $T$ be an arbitrary positive constant and suppose that $T_2<T$. The dynamics in (\ref{bl3b11}) and  (\ref{bl3c11}) guarantee that if $\gamma_i>c_{6}(T,\alpha,\beta)$, and  $T_5\triangleq\max_{i\in\mathcal{N}, k\in\mathcal{N}_i}\{ \frac{1}{2(\gamma_i-c_{6}(T,\alpha, \beta))}\sum_{j=1}^{n_i}\sum_{r=1}^{n_i} a_{i,jr} \Vert\mathbf{y}_{ij}^k(0)-\mathbf{y}_{ir}^k(0) \Vert_\infty\} <T$, we have (a) for all $0 \leq t\leq T$,  	$\Vert\frac{d}{dt}(n_i{\nabla_{\mathbf{x}_{ik}} f_{ij}(\mathbf{z}_{ij})}+\sum_{s=1}^p\frac{n_i}{q_{is}}\delta_{ijs}\lambda_s\nabla_{\mathbf{x}_{ik}} g_{s} (\mathbf{z}_{ij}))\Vert_\infty< c_{6}(T,\alpha,\beta)$; (b) for all $T_6\triangleq\max\{T_2,T_5\}\leq t\leq T$, 	$\mathbf{y}_{ij}^k(t)=\sum_{r=1}^{n_i}{\nabla_{\mathbf{x}_{ik}} f_{ir}(\mathbf{x})}+\sum_{s=1}^p\lambda_s\nabla_{\mathbf{x}_{ik}} g_{s} (\mathbf{x})$.
	\end{lemma}
	
	\begin{proof}
		It is similar to the proof of Lemma \ref{lem}, and thus is omitted here.
	\end{proof}

	Thus, for all $T_6\leq t\leq T$, $\dot{V}\leq 0$, and $V(t)\leq V(T_2)\leq c_b+c_cT_2^2<c_b+c_cT^2$. 
	
	Next, we can prove by contradiction that $\Vert\frac{d}{dt}(n_i{\nabla_{\mathbf{x}_{ik}} f_{ij}(\mathbf{z}_{ij})}+\sum_{s=1}^p\frac{n_i}{q_{is}}\lambda_s\nabla_{\mathbf{x}_{ik}} g_{s} (\mathbf{z}_{ij}))\Vert_\infty
	<c_6(T,\alpha,\beta)$ for all $t\geq T$, and  all $i\in\mathcal{N}$, $j\in\mathcal{N}_i$.
	
	Suppose that there exists a time $T'>T$ such that $\Vert\frac{d}{dt}(n_i{\nabla_{\mathbf{x}_{ik}} f_{ij}(\mathbf{z}_{ij})}+\sum_{s=1}^p\frac{n_i}{q_{is}}\lambda_s\nabla_{\mathbf{x}_{ik}} g_{s} (\mathbf{z}_{ij}))\Vert_\infty
	<c_6(T,\alpha,\beta)$ for all $T\leq t<T'$, and for all $i\in\mathcal{N}$, $j\in\mathcal{N}_i$, and $\Vert\frac{d}{dt}(n_i{\nabla_{\mathbf{x}_{ik}} f_{ij}(\mathbf{z}_{ij})}+\sum_{s=1}^p\frac{n_i}{q_{is}}\lambda_s\nabla_{\mathbf{x}_{ik}} g_{s} (\mathbf{z}_{ij}))\Vert_\infty
	=c_6(T,\alpha,\beta)$  for $t=T'$, and some $i\in\mathcal{N}$, $j\in\mathcal{N}_i$.  According to the definition of $c_6(T,\alpha,\beta)$, we have $V(T')>c_b+c_cT^2$. 
	
	Furthermore, in $[0,T']$,  since $\Vert\frac{d}{dt}(n_i{\nabla_{\mathbf{x}_{ik}} f_{ij}(\mathbf{z}_{ij})}+\sum_{s=1}^p\frac{n_i}{q_{is}}\lambda_s\nabla_{\mathbf{x}_{ik}} g_{s} (\mathbf{z}_{ij}))\Vert_\infty
	\leq c_6(T,\alpha,\beta)$ for all $i\in\mathcal{N}$, $j\in\mathcal{N}_i$,  according to Lemma \ref{l2}, we can obtain that for all $T_6\leq t\leq T'$, 	$\mathbf{y}_{ij}^k(t)=\sum_{r=1}^{n_i}{\nabla_{\mathbf{x}_{ik}} f_{ir}(\mathbf{x})}+\sum_{s=1}^p\lambda_s\nabla_{\mathbf{x}_{ik}} g_{s} (\mathbf{x})$.
	
	Then, for all $T_6\leq t\leq T'$,   $\dot{V}\leq 0$. Furthermore, $V(T')\leq V(T_6)\leq c_b+c_cT_6^2<c_b+c_cT^2$, which contradicts to the above conclusion that $V(T')>c_b+c_cT^2$.
	
	Thus, for all $t\geq 0$, $\Vert\frac{d}{dt}(n_i{\nabla_{\mathbf{x}_{ik}} f_{ij}(\mathbf{z}_{ij})}+\sum_{s=1}^p\frac{n_i}{q_{is}}\delta_{ijs}\lambda_s\nabla_{\mathbf{x}_{ik}} g_{s} (\mathbf{z}_{ij}))\Vert_\infty< c_{6}(T,\alpha,\beta)$. According to Lemma  \ref{l2},  for all $t\geq T_6$, 	$\mathbf{y}_{ij}^k(t)=\sum_{r=1}^{n_i}{\nabla_{\mathbf{x}_{ik}} f_{ir}(\mathbf{x})}+\sum_{s=1}^p\lambda_s\nabla_{\mathbf{x}_{ik}} g_{s} (\mathbf{x})$.
	
	The following analysis is similar to Theorem  \ref{theorem1}, and thus is omitted. 
\end{document}